\documentclass[11pt,a4paper]{article}

\usepackage[english]{babel}
\usepackage[utf8]{inputenc}

\usepackage[colorlinks, linkcolor=blue, citecolor=blue, urlcolor=blue,
breaklinks=true]{hyperref}

\usepackage[DIV=12]{typearea}

\usepackage[affil-it]{authblk}

\usepackage{xparse}
\usepackage{amsmath, amssymb, amsfonts, amsthm}
\usepackage{mathtools}
\usepackage{enumitem}

\usepackage[noadjust]{cite}

\usepackage[capitalise]{cleveref}

\crefname{equation}{}{}

\newcommand{\R}{\mathbb{R}}

\newcommand{\N}{\mathbb{N}}

\newcommand{\ee}{\mathrm{e}}

\DeclareDocumentCommand\dd{ o g d() }{
  \IfNoValueTF{#2}{
    \IfNoValueTF{#3}
    {\mathrm{d}\IfNoValueTF{#1}{}{^{#1}}}
    {\mathinner{\mathrm{d}\IfNoValueTF{#1}{}{^{#1}}\argopen(#3\argclose)}}
  }
  {\mathinner{\mathrm{d}\IfNoValueTF{#1}{}{^{#1}}#2} \IfNoValueTF{#3}{}{(#3)}}
}

\newcommand{\dx}{\dd{x}}
\newcommand{\dy}{\dd{y}}
\newcommand{\dz}{\dd{z}}

\newcommand{\ds}{\dd{s}}
\newcommand{\dt}{\dd{t}}

\renewcommand{\d}[1]{\dd{#1}}

\newcommand{\ddt}{\frac{\mathrm{d}}{\mathrm{dt}}}

\newcommand{\del}{\partial}

\newcommand{\eps}{\varepsilon}

\renewcommand{\L}{\mathcal{L}} 
\newcommand{\A}{\mathcal{A}}
\newcommand{\B}{\mathcal{B}}
\newcommand{\C}{\mathcal{C}}
\newcommand{\No}{\mathcal{N}}

\def\:{\colon}
\newcommand{\vcc}{\vcentcolon}

\DeclareMathOperator{\sgn}{sgn}

\DeclarePairedDelimiter\abs{\lvert}{\rvert}
\DeclarePairedDelimiter\norm{\Vert}{\rVert}
\DeclarePairedDelimiter\scalar{\langle}{\rangle}

\usepackage{color}

\usepackage{ragged2e}
\usepackage{url}

\usepackage{xstring}

\theoremstyle{plain}
\newtheorem{theorem}{Theorem}[section]
\newtheorem{lemma}[theorem]{Lemma}
\newtheorem{proposition}[theorem]{Proposition}
\newtheorem{corollary}[theorem]{Corollary}

\theoremstyle{definition}
\newtheorem{definition}[theorem]{Definition}

\theoremstyle{remark}
\newtheorem{remark}[theorem]{Remark}

\numberwithin{equation}{section}

\title{The scaling hypothesis for Smoluchowski's coagulation equation with bounded
  perturbations of the constant kernel}

\author{
  José A. Cañizo\footnote{\textsc{Departamento de
      Matemática Aplicada, Universidad de Granada, 18071 Granada,
      Spain.}
    \\
    \textit{Email addresses}:
    \texttt{\href{mailto:canizo@ugr.es}{canizo@ugr.es}},
    \texttt{\href{mailto:throm@correo.ugr.es}{throm@correo.ugr.es}}}
  \qquad \qquad
  Sebastian Throm$^*$
}

\date{October 2019}

\begin{document}

\maketitle

\begin{abstract}
  We consider Smoluchowski's coagulation equation with a kernel of the
  form $K = 2 + \epsilon W$, where $W$ is a bounded kernel of
  homogeneity zero. For small $\epsilon$, we prove that solutions
  approach a universal, unique self-similar profile for large times,
  at almost the same speed as the constant kernel case (the speed is
  exponential when self-similar variables are considered). All the
  constants we use can be explicitly estimated. Our method is a
  constructive perturbation analysis of the equation, based on
  spectral results on the linearisation of the constant kernel
  case. To our knowledge, this is the first time the scaling
  hypothesis can be fully proved for a family of kernels which are not
  explicitly solvable.
\end{abstract}

\tableofcontents

\section{Introduction}\label{Sec:Introduction}

We study the long-time behaviour of solutions to Smoluchowski's
coagulation equation, which reads
\begin{equation}
  \label{eq:Smol}
  \del_{\tau}\phi(\tau, \xi)
  = \frac{1}{2}\int_{0}^{\xi} K(\xi-\eta,\eta) \phi(\tau,
  \xi-\eta)\phi(\tau, \eta)\dd{\eta}
  - \phi(\tau, \xi)\int_{0}^{\infty}K(\xi,\eta)\phi(\tau, \eta)\dd{\eta}.
\end{equation}
This equation is a well-known model for coagulation processes in
several contexts such as aerosol dynamics \cite{Fri00,PrK10},
aggregation in planetary formation \cite{SiT79,AlB95} and biology
\cite{AcF97,Ack97}. The unknown $\phi = \phi(\tau, \xi) \geq 0$
represents the density of clusters of size $\xi > 0$ at time
$\tau \geq 0$, and $K = K(\xi, \eta) = K(\eta, \xi) \geq 0$ is the
symmetric \emph{coagulation kernel} giving the coagulation rate of
clusters of size $\xi$ with clusters of size $\eta$. We always
consider the continuum version of this equation, so the size $\xi$ can
take any positive value. A long-standing conjecture is that all
(finite-mass, suitably decaying) solutions to \eqref{eq:Smol} approach
a universal self-similar shape as time $\tau \to +\infty$, as long as
$K(\xi,\eta)$ is a homogeneous function of homogeneity degree
$\gamma \leq 1$ (i.e.,
$K(\lambda \xi,\lambda\eta)=\lambda^{\gamma}K(\xi,\eta)$ for all
$\lambda, \xi,\eta\in(0,\infty)$); this is known as the \emph{scaling
  hypothesis}. More precisely, one expects that there exists a
\emph{self-similar profile} $G$ and a scaling function
$s(\tau)\to\infty$ as $\tau\to\infty$ such that
\begin{equation}\label{eq:scal:hyp}
  \bigl(s(\tau)\bigr)^2\phi\bigl(\tau, s(\tau)\xi \bigr)
  \longrightarrow
  G(\xi) \qquad \text{as }\tau\to\infty,
\end{equation}
in a suitable sense to be determined. This was established in the
particular cases $K(\xi, \eta) = 2$ (constant) and
$K(\xi, \eta) = \xi + \eta$ (linear) in \cite{Menon2004Approach,
  MR2139564} in the sense of weak convergence, with explicit rates
given in \cite{CMM10,doi:10.1137/090759707}. Convergence in stronger
norms for the constant kernel was also found in \cite{CMM10}. There is
also a theory of \emph{fat-tailed} profiles, which represent the
asymptotic behaviour of solutions with slowly decaying tails. We do
not consider them in this work, and we refer the reader to
\cite{Menon2004Approach} for explicitly solvable kernels, and to
\cite{NiV13a,NTV16a,Thr17a} for results on existence and uniqueness of
fat-tailed self-similar profiles with infinite mass.

In this paper we are able to prove the scaling hypothesis in the
regime of finite mass, with an explicit rate, for small bounded
perturbations of the constant kernel. That is, we consider kernels of
the type
\begin{equation}
  \label{eq:h1}
  K = K_\eps(\xi, \eta) = 2 + \eps W(\xi, \eta),
\end{equation}
where $\eps > 0$ and the function
$W \: (0,+\infty) \times (0,+\infty) \to \R$ must be continuous, symmetric in
$\xi,\eta$, satisfy the bound
\begin{equation}
  \label{eq:h2}
  0\leq W(\xi, \eta) \leq 1
  \qquad \text{for all $\xi, \eta > 0$,}
\end{equation}
and be homogeneous of degree zero:
\begin{equation}
  \label{eq:h3}
  W(\lambda \xi, \lambda \eta) = W(\xi, \eta)
  \qquad \text{for all $\xi, \eta, \lambda > 0$.}
\end{equation}
Suitable examples of $W$ include
$$W(\xi, \eta) = \Psi(\xi^\alpha \eta^{-\alpha} + \xi^{-\alpha}
\eta^{\alpha}),$$ where $\alpha \in \R$ and
$\Psi \: (0,+\infty) \to [-1,1]$ is any continuous function. Averages
of functions of this type for different $\alpha$ are also examples of
coefficients satisfying \eqref{eq:h1}--\eqref{eq:h3}.
\begin{remark}
  The choice of $K_0 = 2$ is made for convenience, since by simple
  scaling arguments one can consider perturbations of any constant
  kernel (see Section~\ref{Sec:scale:invariance}). Consequently, also
  the lower bound in~\eqref{eq:h2} can be slightly weakened, i.e.\@
  our result also holds for perturbations which may change sign,
  satisfying $\abs{W(x,y)}\leq 1$. In fact, replacing the constant
  kernel by $2-\norm{W}_{L^{\infty}}$ and the perturbation by
  $\widetilde{W}=W+\norm{W}_{L^{\infty}}$ the assumption~\eqref{eq:h2}
  is satisfied.
\end{remark}

We prove that for $\eps$ small enough, solutions to \eqref{eq:Smol}
approach a unique, universal self-similar profile at an explicit
algebraic rate (which becomes exponential when self-similar variables
are considered; see below), in the sense of the $\|\cdot\|_{L^1_k}$ norm
defined by
\begin{equation*}
  \| f \|_{L^1_k} := \int_0^\infty |f(x)| (1+x)^k \dx.
\end{equation*}
Our main result is summarised in the following theorem:
\begin{theorem}
  \label{thm:main-intro}
  Let $K = K_\eps$ be a bounded perturbation of the constant
  kernel satisfying \eqref{eq:h1}, \eqref{eq:h2} and \eqref{eq:h3}.
  \begin{enumerate}
  \item There exists $\eps_1 > 0$ such that for
    $0 \leq \eps \leq \eps_1$ there exists a unique self-similar
    profile $G_\eps$ with unit mass.

  \item Given $R > 0$ and $k > 1$, there exists
    $0 < \eps_3 \leq \eps_1$ (depending only on $R$ and $k$) and $M$
    (depending only on $k$) such that for $0 \leq \eps < \eps_3$ any
    solution $\phi$ to the Smoluchowski equation \eqref{eq:Smol} with
    nonnegative initial condition $\phi_0$ such that
    \begin{equation*}
      \int_0^\infty \xi \phi_0(\xi) \d \xi = 1,
      \qquad
      \int_0^\infty \big|\phi_0(\xi) - G_\eps(\xi)\big|
      (1+\xi)^k \d \xi \leq R
    \end{equation*}
    satisfies
    \begin{equation*}
      \| f(t, \cdot) - G_\eps \|_{L^1_k} \leq
      C e^{-\lambda_\eps t} \| f_0 - G_\eps \|_{L^1_k}
      \qquad \text{for all $t \geq 0$,}
    \end{equation*}
    with $\lambda_\eps := \frac{1}{2}-M\eps$, for some $C > 0$
    depending on $k$ and $R$, and where
    $f(t,x) := e^{2t} \phi(e^t-1, x e^t)$. Equivalently,
    \begin{equation*}
      \int_0^\infty
      \Big| (\tau+1)^2\phi(\tau, (\tau+1)x) - G_\eps(x)\Big| (1+x)^k \d x
      \leq
      C (1+\tau)^{-\lambda_\eps} \| \phi_0 - G_\eps \|_{L^1_k}
    \end{equation*}
    for all $\tau \geq 0$.
  \end{enumerate}
  All constants appearing in this theorem can be explicitly estimated.
\end{theorem}

As far as we know, this is the first time the scaling hypothesis can
be proved to hold for kernels which do not allow for an explicit
solution of equation \eqref{eq:Smol}. Results for the so-called
\emph{diagonal} kernels were obtained by \cite{LNV18}, and in this
case the approach to self-similarity does not happen for all initial
conditions. The part of our result on uniqueness of the profiles is
not new: it has been proved for more general perturbations by Laplace
transform methods in \cite{NV2014,NTV15,NTV16} and via compactness
arguments in $L^1$ in \cite{Thr19}. An improvement here is that we are
able to give an explicit estimate of $\eps_0$ in the case of bounded
perturbations, and explicit estimates on the closeness of $G_\eps$ to
$G_0 = e^{-x}$, the self-similar profile for the equation with
$K = 2$. For kernels $K$ with negative homogeneity degree it has been
recently proved that self-similar profiles are unique \cite{Lau18}.

The strategy to prove our results is a perturbation argument, carried
out in a constructive way, using the fact that the case of constant
coefficients is fairly well understood. This has been done for kinetic
equations involving the Boltzmann operator in
\cite{Mischler2006Cooling, MR2264623,MM09, Canizo2015Exponential}, but
has not been done for coagulation-type equations as far as we
know. Results on convergence to equilibrium for the Becker-Döring
equation were developed in \cite{Canizo2013Exponential} using
properties of the linearised operator, with techniques similar to
those in Section \ref{sec:convergence}. We need three main ingredients
in order to complete our perturbation arguments:

\begin{enumerate}
\item First, we need a good global exponential convergence result for
  the constant coefficients case. Convergence without rate is known
  since \cite{Menon2004Approach}, and in order to obtain rates one can
  use results in \cite{CMM10,doi:10.1137/090759707}. It turns out that
  a global convergence in an $L^1$ or $L^2$ space is more convenient
  for us, so we use a refined version of the results in
  \cite{CMM10}. These results are given in \cref{sec:constant_kernel},
  with precise estimates on the dependence of the constants given in
  \cref{sec:aux_thm}.
  
\item One also needs results on the stability of self-similar profiles
  with respect to the perturbation; that is, we need to show that a
  profile $G_\eps$ associated to the perturbed kernel $K_\eps$
  must be close to the unique profile $G_0$ for the constant
  coefficient case. In the case of perturbations by a bounded
  coagulation coefficient, it turns out that the perturbation of the
  operator is continuous in weighted $L^1$ norms, so we are led to
  work in these spaces. This kind of stability results was studied in
  \cite{Thr19}, and we are able to give a new proof with explicit
  estimates in \cref{sec:bound_stability}.
  
\item Finally, we need to show that the linearised equation around the
  self-similar profile for the constant case has a spectral gap. More
  importantly, we need to do this in a norm which allows us to
  complete the perturbation argument, so again we are forced to work
  in weighted $L^1$ norms. A spectral gap in these spaces is proved in
  \cref{sec:gap_L1}, using results from Sections
  \ref{sec:lin_operator} and \ref{Sec:tools:sg}.
  
\end{enumerate}

The paper is organised as follows: Section \ref{sec:prelims} gathers
some preliminary results which are known or can be obtained almost
directly from existing results. In Section \ref{sec:bound_stability}
we give bounds on self-similar profiles (some of which are new) and
show a quantitative stability result in weighted $L^1$ norms. The
result itself is not new, but we give a new proof that makes it fully
quantitative. Sections \ref{sec:lin_operator}--\ref{sec:gap_L1} study
the linearised operator and show it has a spectral gap in the weighted
$L^1$ spaces we need. Finally, Sections \ref{sec:local} and
\ref{sec:convergence} use all of this to show uniqueness and
exponential stability of self-similar profiles for small values of the
perturbation parameter $\eps$.

\section{Preliminaries}
\label{sec:prelims}

\subsection{Self-similar change of variables}

By scaling arguments and mass conservation we obtain that in the case
of kernels of homogeneity zero the function $s(\tau)$
in~\eqref{eq:scal:hyp} is given by $s(\tau)=(1+\tau)$, up to a time
shift. Thus, plugging the self-similar ansatz
\begin{equation*}
  \phi(\tau, \xi)=(1+\tau)^{-2}f\bigl(\log(1+\tau),(1+\tau)^{-1} \xi \bigr)
\end{equation*}
into~\eqref{eq:Smol} we obtain
\begin{multline}\label{eq:Smol:selfsim}
  \del_{t}f(t, x)=\frac{1}{2}\int_{0}^{x}K(x-y,y)f(x-y,t)f(t, y)\dy\\*
  -f(t, x)\int_{0}^{\infty}K(x,y)f(t, y)\dy+2f(t, x)+x\del_{x}f(t, x).
\end{multline}
To simplify the notation, let us define
the operator
\begin{equation}\label{eq:coag:op}
  \C_{K}(f,f)\vcc=\frac{1}{2}\int_{0}^{x}K(x-y,y)f(x-y)f(y)\dy
  -f(x)\int_{0}^{\infty}K(x,y)f(y)\dy,
\end{equation}
which motivates the definition of the following symmetric bilinear
form which will be useful later:
\begin{multline}\label{eq:sym:coag:op}
  \C_{K}(g,h)\vcc=\frac{1}{2}\int_{0}^{x}K(x-y,y)g(x-y)h(y)\dy\\*
  -\frac{1}{2}g(x)\int_{0}^{\infty}K(x,y)h(y)\dy-\frac{1}{2}h(x)\int_{0}^{\infty}K(x,y)g(y)\dy.
\end{multline}
When the kernel $K$ is $K_\eps = 2 + \eps W$, as it is almost
always the case in this paper, we will write
$\C_{K_\eps} \equiv \C_\eps$. We can then write equation
\eqref{eq:Smol:selfsim} in an abbreviated form as
\begin{equation*}
  \del_{t}f=\C_{K}(f,f)+2f+x\del_{x}f.
\end{equation*}
We refer to this equation as the Smoluchowski equation in
self-similarity variables, or simply the \emph{self-similar
  Smoluchowski equation}.
  
To simplify the notation at some places (especially when the kernel $K$ is constant), we may also use the following notation $(f\ast g)(x)=\int_{0}^{x}f(x-y)g(y)\dy$ for the convolution.

\subsection{Scale invariances}\label{Sec:scale:invariance}

We collect here some elementary properties
about~\eqref{eq:Smol:selfsim}. It is well-known
that~\eqref{eq:Smol:selfsim} preserves the total mass
$m_{1}(t)\vcc=\int_{0}^{\infty}xf(x,t)\dx$, i.e.\@
$m_{1}(t)\equiv m_{1}(0)$, provided that the kernel $K$ grows at most
linearly at infinity. Yet, for kernels with superlinear growth, a loss
of total mass in finite time occurs which is known as gelation (e.g.\@
\cite{EMP02}).

Furthermore, if $f$ is a solution
to~\eqref{eq:Smol:selfsim} with kernel $K$ one easily checks that for any $\alpha > 0$ the
function $g = \frac{1}{\alpha} f$ solves the self-similar Smoluchowski equation with kernel $\alpha K$, i.e.\@
\begin{equation*}
  \del_{t}g=\C_{\alpha K}(g,g)+2g+x\del_{x}g.
\end{equation*}
Moreover, one verifies that for each solution $f$
of~\eqref{eq:Smol:selfsim} also the rescaled function
$f_{a}\vcc=af(ax)$ is a solution to~\eqref{eq:Smol:selfsim} with the
same kernel. Note also that for both transformations the mass
changes. Summarising, we find that for $f$
solving~\eqref{eq:Smol:selfsim} with kernel $K$, the function
$h(x)=\frac{a}{\alpha}f(ax)$ is a solution to~\eqref{eq:Smol:selfsim}
with $K$ replaced by $\alpha K$. Moreover, if $f$ has total mass
$m_{1}$, i.e.\@ $\int_{0}^{\infty}xf(x)\dx=m_{1}$ we get for $h$ that
\begin{equation*}
  \int_{0}^{\infty}xh(x)\dx=\frac{a}{\alpha}\int_{0}^{\infty}xf(ax)\dx=\frac{1}{a\alpha}\int_{0}^{\infty}xf(x)\dx=\frac{m_{1}}{a\alpha}.
\end{equation*}
These computations allow to transform solutions
to~\eqref{eq:Smol:selfsim} for different constant kernels and modify
the total mass. As a consequence, we can assume without loss of
generality in the following that the constant kernel $K$ is
given by $K=2$ and the total mass of the solutions and profiles is
given by one.

\begin{remark}
  Note also that in~\cite{CMM10} the kernel was chosen to be
  $K\equiv 1$. However, by the considerations above, all results can
  be easily rescaled to the case $K=2$ which is what we will always
  implicitly do during this work.
\end{remark}

\subsection{Function spaces}\label{Sec:function:spaces}

We collect in this section the function spaces and corresponding
notation which we use throughout this work. If nothing else is stated,
all functions live on the set $(0,\infty)$. First, for a general
weight function $w\colon (0,\infty)\to (0,\infty)$ and
$p\in[0,\infty)$ we define in the usual way the weighted $L^p$ space
\begin{equation*}
  L^{p}(w)\vcc=\biggl\{f\in L^{p}\;\bigg|\; \int_{0}^{\infty}\abs{f(x)}^{p}w(x)\dx<\infty\biggr\}
\end{equation*}
with norm $\norm{f}_{L^{p}(w)}\vcc=\left(\int_{0}^{\infty}\abs{f(x)}^{p}w(x)\dx\right)^{1/p}$. The most important case for this work will be the choice $p=1$ and $w(x)=(1+x)^{k}$ with $k\geq 1$ which gives
\begin{equation*}
  L^{1}((1+x)^{k})\vcc=\biggl\{f\in L^{1}\;\bigg|\; \int_{0}^{\infty}\abs{f(x)}(1+x)^{k}\dx<\infty\biggr\}
\end{equation*}
with corresponding norm
$\norm{f}_{L^{1}((1+x)^{k})}\vcc=\int_{0}^{\infty}\abs{f(x)}(1+x)^{k}\dx$. To
simplify the notation at some places, we might also write
$L^{1}_{k}=L^{1}((1+x)^{k})$ and we might use the abbreviations
$\norm{\cdot}_{k}=\norm{\cdot}_{L^{1}_{k}}=\norm{\cdot}_{L^{1}((1+x)^{k})}$.
For parts of this work, we also need several spaces with rather weak
norms defined via the primitive.

In particular, we define the norm
\begin{equation*}
  \| h \|_{W^{-1,\infty}} := \sup_{x > 0}  \left| \int_x^\infty h(y) \dy  \right|.
\end{equation*}

Moreover, we introduce the (weighted) $L^2$ space
$H^{-1}(\ee^{\mu x})$ which arose naturally in \cite{CMM10}. More
precisely, for $\mu\geq 0$ we define the norm
\begin{equation*}
  \norm{h}_{H^{-1}(\ee^{\mu x})}^2=\int_{0}^{\infty}(D^{-1}h(y))^2\ee^{\mu y}\dy.
\end{equation*}
where $D^{-1}h(y)=\int_{y}^{\infty}h(x)\dx$ denotes the
primitive. This norm originates from the following scalar product
\begin{equation*}
  \scalar{g,h}_{H^{-1}(\ee^{\mu x})}=\int_{0}^{\infty}(D^{-1}g)(x)(D^{-1}h)(x)\ee^{\mu x}\dx. 
\end{equation*}
The corresponding (Hilbert) space
\begin{equation*}
  W^{-1,2}(\ee^{\mu x})=H^{-1}(\ee^{\mu x})
\end{equation*}
is then given as the completion of $C_{c}^{\infty}(0,\infty)$ with
respect to the norm $\norm{\cdot}_{H^{-1}(\ee^{\mu x})}$. For $\mu>0$,
we additionally introduce the corresponding subspace given by the
constraint of zero first moment, i.e.\@
\begin{equation*}
  H^{-1}(\ee^{-\mu x})\cap \left\{\int_{0}^{\infty}xh(x)\dx=0\right\}
\end{equation*}
is defined as the completion of
$(C_{c}^{\infty}(0,\infty)\cap \{\int_{0}^{\infty}xh(x)\dx=0\})$ with
respect to $\norm{\cdot}_{H^{-1}(\ee^{\mu x})}$.

\begin{remark}
  The latter definition is justified by the following estimate which
  exploits integration by parts as well as Hölder's inequality: for
  $h\in C_{c}^{\infty}(0,\infty)$ we have
  \begin{multline*}
    \int_{0}^{\infty}xh(x)\dx=-\int_{0}^{\infty}x\del_{x}\biggl(\int_{x}^{\infty}h(z)\dz\biggr)\dx=\int_{0}^{\infty}\ee^{-\frac{\mu}{2}x}\ee^{\frac{\mu}{2}x}\biggl(\int_{x}^{\infty}h(z)\dz\biggr)\dx\\*
    \leq \biggl(\int_{0}^{\infty}\ee^{-\mu x}\dx\biggr)^{1/2}\norm{h}_{H^{-1}(\ee^{\mu x})}=\mu^{-1/2}\norm{h}_{H^{-1}(\ee^{\mu x})}.
  \end{multline*}
  Thus, by density the integral $\int_{0}^{\infty}xh(x)\dx$ is
  meaningful for all $h\in H^{-1}(\ee^{\mu x})$.
\end{remark}

We also note the following continuous embeddings which will be
especially important for this work.

\begin{lemma}\label{Lem:cont:emb:L2:L1}
  For each $k\geq 0$ and $\mu>0$ the space $L^{2}(\ee^{\mu x})$ embeds
  continuously into $L^{1}((1+x)^{k})$.
\end{lemma}

\begin{proof}
  Using the splitting $1=\ee^{-\mu x/2}\ee^{\mu x/2}$ and Hölder's inequality we find
  \begin{multline*}
    \norm{h}_{L^1_k}=\int_{0}^{\infty} \abs{h(x)}(1+x)^{k}\dx=\int_{0}^{\infty}\abs{h(x)}\ee^{\frac{\mu}{2}x}\ee^{-\frac{\mu}{2}x}(1+x)^{k}\dx\\*
    \leq \biggl(\int_{0}^{\infty}\abs{h(x)}^{2}\ee^{\mu x}\dx\biggr)^{1/2}\biggl(\int_{0}^{\infty}\ee^{-\mu x}(1+x)^{2k}\dx\biggr)^{1/2}\leq C(\mu, k)\norm{h}_{L^{2}(\ee^{\mu x})}.
  \end{multline*} 
\end{proof}

\begin{lemma}\label{Lem:cont:emb:L2:spaces}
  For each $\mu>0$ the space $L^2(\ee^{\mu x})$ embeds continuously
  into $H^{-1}(\ee^{\mu x})$.
\end{lemma}

\begin{proof}
  It suffices to verify the embedding for the dense subset
  $C_{c}^{\infty}(0,\infty)$. In this case, integration by parts
  yields
  \begin{multline*}
    \norm{h}_{H^{-1}(\ee^{\mu x})}^{2}
    = \int_{0}^{\infty}\ee^{\mu x}\biggl(\int_{x}^{\infty}h(z)\dz\biggr)^{2}\dx
    = \frac{1}{\mu}\int_{0}^{\infty}\del_{x}(\ee^{\mu x})
    \biggl(\int_{x}^{\infty}h(z)\dz\biggr)^{2}\dx
    \\*
    = -\frac{1}{\mu}\biggl(\int_{0}^{\infty}h(z)\dz\biggr)^{2}
    +\frac{2}{\mu}\int_{0}^{\infty}\ee^{\mu x}
    \biggl(\int_{x}^{\infty}h(z)\dz\biggr) h(x)\dx.
  \end{multline*}
  Note, that we additionally exploit here that the boundary term at
  infinity vanishes since $h\in C_{c}^{\infty}$. To proceed, we use
  that the first expression on the right-hand side is non-positive
  while the second one can be estimated by means of Hölder's
  inequality together with the splitting
  $\ee^{\mu x}=\ee^{\frac{\mu}{2}x}\ee^{\frac{\mu}{2}x}$ which results
  in
  \begin{equation*}
    \norm{h}_{H^{-1}(\ee^{\mu x})}^{2}\leq \frac{2}{\mu}\norm{h}_{H^{-1}(\ee^{\mu x})}\norm{h}_{L^{2}(\ee^{\mu x})}.
  \end{equation*}
  Cauchy's inequality with parameter $2/\mu$ finally yields
  \begin{equation*}
    \norm{h}_{H^{-1}(\ee^{\mu x})}^{2}
    \leq \frac{1}{2}\norm{h}_{H^{-1}(\ee^{\mu x})}^{2}
    +\frac{2}{\mu^2}\norm{h}_{L^{2}(\ee^{\mu x})}^{2}
  \end{equation*}
  which finishes the proof by providing
  $\norm{h}_{H^{-1}(\ee^{\mu x})}^{2}\leq 4/\mu^{2}
  \norm{h}_{L^{2}(\ee^{\mu x})}^{2}$.
\end{proof}

\subsection{Continuity properties of \texorpdfstring{$\C_{K}$}{CK}}

The spaces $L^1((1+x)^k)$ are convenient to work with because, for a
bounded kernel $K$, the coagulation operator $C_K$ is a continuous
bilinear form on them, as we show in the next proposition. Notice that this
is not true for weighted $L^2$ spaces, for example.

\begin{proposition}\label{Prop:CK:cont}
  For $K\colon (0,\infty)^{2}\to (0,\infty)$ bounded, the bilinear
  form $\C_{K}$ given by~\eqref{eq:sym:coag:op} is continuous from
  $L^{1}((1+x)^{k})$ to itself for each $k\geq 0$ and we have
  \begin{equation*}
    \norm{\C_{K}(g,h)}_{L^1_k}\leq \frac{3}{2}\norm{K}_{L^{\infty}}\norm{g}_{L^1_k}\norm{h}_{L^1_k}.
  \end{equation*}
\end{proposition}

\begin{proof}
  From~\eqref{eq:sym:coag:op} together with Fubini's theorem we find
  \begin{multline*}
    \norm{\C_{K}(g,h)}_{L^1_k}\leq \frac{1}{2}\int_{0}^{\infty}\int_{0}^{\infty}\abs*{K(x,y)}\abs{g(x)}+\abs{h(y)}\bigl(1+(x+y)\bigr)^{k}\dx\dy\\*
    +\frac{1}{2}\int_{0}^{\infty}\int_{0}^{\infty}\abs{K(x,y)}\abs{g(x)}\abs{h(y)}(1+x)^{k}\dx\dy+\frac{1}{2}\int_{0}^{\infty}\int_{0}^{\infty}\abs{K(x,y)}\abs{g(y)}\abs{h(x)}(1+x)^{k}\dx\dy.
  \end{multline*}
  Exploiting that $(1+(x+y))^{k}\leq (1+x)^{k}(1+y)^{k}$ and
  $1\leq (1+x)^{k}$ if $x,y\in (0,\infty)$ and $k\geq 0$, we further get
  \begin{multline*}
    \norm{\C_{K}(g,h)}_{L^1_k}\leq\frac{3}{2}\norm{K}_{L^{\infty}}\biggl(\abs{g(x)}(1+x)^{k}\dx\biggr)\biggl(\int_{0}^{\infty}\abs{h(y)}(1+y)^{k}\dy\biggr)\\*
    =\frac{3}{2}\norm{K}_{L^{\infty}}\norm{g}_{L^1_k}\norm{h}_{L^1_k}.
  \end{multline*}
\end{proof}

\subsection{Existence of time-dependent solutions and moment estimates}

Following~\cite{EMR05} we will use the following concept of (mild) solutions to~\eqref{eq:Smol:selfsim}.

\begin{definition}
  For $k\leq 1$ let $f_{0}\in L_{k}^{1}$. A function
  $f\in C([0,\infty),L_{k}^{1})$ is denoted a (mild) solution
  to~\eqref{eq:Smol:selfsim} if it satisfies
  \begin{equation}\label{eq:def:mild}
    f=S_{t}f_{0}+\int_{0}^{t}S_{t-s}\C_{K}(f,f)(s)\ds
  \end{equation}
  where $(S_{t})_{t\geq 0}$ is the semigroup generated by the operator $h\mapsto 2h+xh'$, i.e.\@ $(S_{t}h)(x)=\ee^{2t}h(\ee^{t}x)$.
\end{definition}

\begin{remark}\label{Rem:def:sol:alt}
  According to \cite{EMR05}, the following two solution concepts are equivalent to mild solutions if $f\in C([0,\infty),L_{k}^{1})$:
  \begin{enumerate}
  \item $f$ is a \emph{weak} or \emph{distributional} solution, i.e.\@
    it satisfies
    \begin{equation*}
      \int_{0}^{T}\int_{0}^{\infty}\Bigl(f\del_{t}\varphi+\C_{K}(f,f)\varphi-f(2\varphi -\del_{x}(x\varphi)\Bigr)\dx\dt=\int_{0}^{\infty}f_{0}\varphi(0,\cdot)\dx
    \end{equation*}
    for all $T>0$ and $\varphi\in C_{c}^{1}([0,T)\times (0,\infty))$.
  \item $f$ is a renormalised solution, i.e.\@ it satisfies
    \begin{equation*}
      \frac{\dd}{\dt}\int_{0}^{\infty}\beta(f)\varphi\dx=\int_{0}^{\infty}\C_{K}(f,f)\beta'(f)\varphi-f(2\varphi-\del_{x}(x\varphi)\dx
    \end{equation*}
    in the sense of distributions in $[0,\infty)$ for all $\beta\in C^{1}(\R)\cap W^{1,\infty}(\R)$ and $\varphi\in C^{1}(0,\infty)\cap L^{\infty}(0,\infty)$.
  \end{enumerate}
\end{remark}

The self-similar profiles are then seen to be the stationary solutions
to~\eqref{eq:Smol:selfsim}.

\begin{definition}\label{Def:profile}
  A function $G\in L_{1}^{1}$ is a
  \emph{self-similar profile} of~\eqref{eq:Smol} if it is a stationary
  solution to~\eqref{eq:Smol:selfsim}, i.e.\@ it is a fixed-point
  for~\eqref{eq:def:mild}.
\end{definition}

\begin{remark}
  Due to Remark~\ref{Rem:def:sol:alt} we see that $G$ is a self-similar if it is a weak stationary solution to~\eqref{eq:Smol:selfsim}, i.e.\@ with left-hand side zero. 
\end{remark}

\begin{remark}\label{Rem:profile:integrated}
  Note that when dealing with self-similar profiles frequently an integrated version of (the stationary) equation~\eqref{eq:Smol:selfsim} is used, i.e.\@ $x^2 G(x)=\int_{0}^{x}\int_{x-y}^{\infty}yK(y,z)G(y)G(z)\dz\dy$ (e.g.\@ \cite{NiV14,NTV15,NTV16}). However, in view of \cite[Lemma~2.11]{EsM06} together with Proposition~\ref{Prop:L1:est:profile} both formulations are equivalent.
\end{remark}

The next proposition is a classical result which provides the existence of solutions to~\eqref{eq:Smol:selfsim} under the assumption of a bounded coagulation coefficient. A proof can be found for example in~\cite[Lemma~2.8]{EMR05}

\begin{proposition}\label{Prop:self:sim:evolution:existence}
  Let $K$ be a bounded, symmetric and homogeneous of degree zero. For $k\geq 1$ let $f_{0}\in L^{1}((1+x)^{k})$. Then there exists a unique solution $f\in L^{1}((1+x)^{k})$ to~\eqref{eq:Smol:selfsim} on $(0,\infty)$ with $f(0,\cdot)=f_{0}$.  Moreover, there exists $C_{k}$, which only depends on $\norm{f_{0}}_{L^{1}_{k}}$ increasingly, such that
  \begin{equation*}
    \sup_{t\geq 0}\norm{f(t,\cdot)}_{L^1_k}\leq C_{k}.
  \end{equation*}
\end{proposition}

\begin{remark}
  Note that the bound on the norm provided by~\cite[Lemma~2.8]{EMR05}
  is local in time. However, arguing analogously as in the proof of
  Proposition~\ref{Prop:L1:est:profile} one can easily verify that the
  estimate holds globally as stated above.
\end{remark}

The next statement provides a more explicit estimate on the integral
of solutions to~\eqref{eq:Smol:selfsim} for bounded perturbations of
the constant kernel.

\begin{proposition}\label{Prop:bound:m0}
  Let $W$ satisfy \cref{eq:h1,eq:h2,eq:h3} and let $f_{\eps}$ be a
  solution to~\eqref{eq:Smol:selfsim} with kernel $K_{\eps}=2+\eps W$
  and denote $m_{0}(t)\vcc=\int_{0}^{\infty}f_{\eps}(x,t)\dx$. Then
  the estimate
  \begin{equation*}
    m_{0}(t)\leq \frac{\ee^{t}}{\frac{1}{m_{0}(0)}+\ee^{t}-1}\leq \max\{m_{0}(0),1\}
  \end{equation*}
  holds for all $t\geq 0$.
\end{proposition}

\begin{proof}
  It is well-known that $\C_{K}$ satisfies the relation
  \begin{equation*}
    \int_{0}^{\infty}\C_{K}(f,f)(x)\dx=-\frac{1}{2}\int_{0}^{\infty}\int_{0}^{\infty}K(x,y)f(x)f(y)\dx\dy
  \end{equation*}
  which follows from Fubini's Theorem. Thus, integrating~\eqref{eq:Smol:selfsim} and noting that integration by parts yields $\int_{0}^{\infty}xf'_{\eps}(x,t)\dx=-m_{0}(t)$ we obtain
  \begin{multline*}
    \del_{t}m_{0}(t)=2m_{0}(t)-m_{0}(t)-\frac{1}{2}\int_{0}^{\infty}\int_{0}^{\infty}K_{\eps}(x,y)f_{\eps}(x)f_{\eps}(y)\dx\dy\\*
    =m_{0}(t)-m_{0}^{2}(t)-\frac{\eps}{2}\int_{0}^{\infty}\int_{0}^{\infty}W(x,y)f_{\eps}(x)f_{\eps}(y)\dx\dy\leq m_{0}(t)-m_{0}^{2}(t).
  \end{multline*}
  In the last step, we exploited that $W,f_{\eps}\geq 0$. Solving this differential inequality explicitly, the claim directly follows.
\end{proof}

\begin{remark}
  The result of the previous proposition holds in general for any
  kernel $K \geq 2$, as long as time-dependent solutions can be proved
  to exist for that kernel. We have stated it for $K = 2 + \eps W$
  since it is the only case used in the rest of this paper.
\end{remark}

\subsection{Stability of time-dependent solutions with respect to
  perturbations}

\Cref{Prop:self:sim:evolution:existence} together with the continuity
results in \cref{Prop:CK:cont} easily implies that on a fixed time
interval solutions to~\eqref{eq:Smol:selfsim} for the perturbed and
unperturbed kernel stay close at order $O(\eps)$:

\begin{lemma}\label{Lem:L1:evolution:close}
  Let $K_{\eps}$ satisfy \cref{eq:h1,eq:h2,eq:h3} and let $f^{\eps}$
  and $f^{0}$ be the solutions to~\eqref{eq:Smol:selfsim} with kernels
  $K_\eps$ and $K_0$, respectively, and with the same initial
  condition $f^{\eps}(0,\cdot)=f^{0}(0,\cdot)$.  Then there exists
  a constant $C > 0$ depending only on $\|f_0\|_{L^1_k}$ such
  that
  \begin{equation*}
    \norm{f^{\eps}(t,\cdot)-f^{0}(t,\cdot)}_{L^{1}_k}
    \leq \eps C (\ee^{C t}-1).
  \end{equation*}
  Also, the constant $C > 0$ depends increasingly on $\|f_0\|_{L^1_k}$.
\end{lemma}

\begin{proof}
  We take the difference of~\eqref{eq:Smol:selfsim} for $\eps$ and $\eps=0$ and multiply by $\sgn(f^{\eps}-f^{0})$ which allows to rewrite
  \begin{multline*}
    \del_{t}\abs{f^{\eps}(t,\cdot)-f^{0}(t,\cdot)}=\bigl(\C_{2+\eps W}(f^{\eps},f^{\eps})-\C_{2}(f^{0},f^{0})\bigr)\sgn(f^{\eps}-f^{0})\\*
    +2\abs{f^{\eps}(t,\cdot)-f^{0}(t,\cdot)}+x\del_{x}\abs{f^{\eps}(t,\cdot)-f^{0}(t,\cdot)}.
  \end{multline*}
  We multiply by $(1+x)^{k}$, integrate over $(0,\infty)$ and use that
  integration by parts allows to rewrite and estimate
  \begin{multline*}
    \int_{0}^{\infty}x(1+x)^{k}\del_{x}\abs{f^{\eps}(t,x)-f^{0}(t,x)}\dx\\*
    =-\norm{f^{\eps}(t,\cdot)-f^{0}(t,\cdot)}_{L^{1}((1+x)^{k}}
    -k\int_{0}^{\infty}x(1+x)^{k-1}\abs{f^{\eps}(t,x)-f^{0}(t,x)}\dx\\*
    \leq -\norm{f^{\eps}(t,\cdot)-f^{0}(t,\cdot)}_{L^{1}_k}.
  \end{multline*}
  Thus, we obtain
  \begin{equation*}
    \ddt \norm{f^{\eps}-f^{0}}_{L^1_k}
    \leq
    \norm{\C_{2+\eps W}(f^{\eps},f^{\eps})-\C_{2}(f^{0},f^{0})}_{L^1_k}
    +\norm{f^{\eps}-f^{0}}_{L^1_k}.
  \end{equation*}
  Notice that we have omitted the variables $(t,\cdot)$ for
  brevity. Using that
  $\C_{2+\eps
    W}(f^{\eps},f^{\eps})-\C_{2}(f^{0},f^{0})=\C_{2}(f^{\eps}-f^{0},f^{\eps}+f^{0})+\eps
  \C_{W}(f^{\eps},f^{\eps})$ we can estimate this further, using
  \cref{Prop:CK:cont}:
  \begin{multline*}
    \ddt \norm{f^{\eps}-f^{0}}_{L^1_k}
    \leq
    \norm{\C_{2}(f^{\eps}-f^{0},f^{\eps}+f^{0})}_{L^1_k}
    +\eps\norm{\C_{W}(f^{\eps},f^{\eps})}_{L^1_k}
    \\*
    \leq 3 \norm{f^{\eps}-f^{0}}_{L^1_k}
    \norm{f^{\eps} + f^{0}}_{L^1_k}
    + \eps \norm{f^\eps}_{L^1_k}^2.
  \end{multline*}
  From \cref{Prop:self:sim:evolution:existence} we know that there exists
  $C_1 > 0$ depending only on $\norm{f^0(0,\cdot)}_{L^1_k}$ (and
  increasingly) such that
  \begin{equation*}
    \norm{f^{\eps}}_{L^1_k} + \norm{f^{0}}_{L^1_k}
    \leq C_1
    \qquad \text{for all $t \geq 0$.}
  \end{equation*}
  Hence,
  \begin{equation*}
    \ddt \norm{f^{\eps}-f^{0}}_{L^1_k}
    \leq
    3 C_1 \norm{f^{\eps}-f^{0}}_{L^1_k}
    + \eps C_1^2.
  \end{equation*}
  From Gronwall's inequality we get
  \begin{equation*}
    \norm{f^{\eps}-f^{0}}_{L^1_k}
    \leq
    \frac{\eps C_1}{3} ( \ee^{3 C_1 t} - 1),
  \end{equation*}
  which implies the statement for $C := 3 C_1$.
\end{proof}

\subsection{Asymptotic behaviour of solutions for the constant kernel}
\label{sec:constant_kernel}

First, we recall from \cite[Lemma 6.1]{CMM10} the following statement
which provides exponential convergence in $L^2$ to the unique profile
for the constant coagulation kernel $K=2$.

\begin{theorem}\label{Thm:L2:convergence}
  Let $f$ be a solution to~\eqref{eq:Smol:selfsim} for the constant
  kernel $K=2$ with total mass one and initial condition $f_0$ such
  that $f_0 \in L^{2}(\dx) \cap L^{1}(x^2\dx)$. Let
  $G^{0}(x)=\ee^{-x}$ be the unique stationary solution
  to~\eqref{eq:Smol:selfsim} with total mass $1$. There exists an
  explicit constant $C > 0$ depending only on $\|f_0\|_{L^2}$ and
  $\|f_0\|_{L^1_2}$ such that
  \begin{equation*}
    \norm{f^{0}(t,\cdot)-G^{0}}_{L^{2}}
    \leq
    C\ee^{-\frac{1}{2}t}
    \qquad \text{for all } t \geq 0.
  \end{equation*}
  Also, the constant $C > 0$ depends increasingly on $\|f_0\|_{L^2}$ and
  $\|f_0\|_{L^1_2}$.
\end{theorem}
The statement we give here is slightly different to that in
\cite[Lemma 6.1]{CMM10} in that we say one can find an explicit
constant $C$ depending only on $\|f_0\|_{L^2}$ and
$\|f_0\|_{L^1_2}$. Although this constant was not specified in
\cite{CMM10}, this can be seen from the proof of the lemma, which
consists on explicit estimates on the Fourier transform. We give a
justification of this in \cref{sec:aux_thm}.

The aim of the following two lemmas is to transfer the convergence in
Lemma~\ref{Thm:L2:convergence} to $L^{1}((1+x)^{k})$. For this, we
first prove an elementary interpolation inequality
(Lemma~\ref{Lem:interpolation1}) which then allows to extend the
convergence in Lemma~\ref{Thm:L2:convergence} to $L^{1}((1+x)^{k})$
(Lemma~\ref{Lem:constant:kernel:L1:conv}).

\begin{lemma}\label{Lem:interpolation1}
  Let $k^*>k$ and $\alpha\in(0,\min\{\frac{2(k^*-k)}{1 + 2 k^*},1\})$
  be given and assume that
  $f\in L^{1}((1+x)^{k^*})\cap L^{2}(0,\infty)$. There exists a
  constant $C$ which depends on $\alpha$, $k$ and $k^*$ such that
  \begin{equation*}
    \norm{f}_{L^{1}_k}
    \leq
    C \norm{f}_{L^{2}}^{\alpha}
    \norm{f}_{L^{1}_{k^*}}^{1-\alpha}.
  \end{equation*}
\end{lemma}

\begin{proof}
  The claim follows from a straightforward application of Hölder's
  inequality with the three exponents $p_1=p_2=2/\alpha$ and
  $p_3=1/(1-\alpha)$. In fact, we have
  \begin{multline*}
    \norm{f}_{L^1_k}
    =
    \int_{0}^{\infty}\abs{f(x)}(1+x)^{k}\dx
    \\
    =
    \int_{0}^{\infty}\abs{f(x)}^{\alpha}
    \abs{f(x)}^{1-\alpha}(1+x)^{(1-\alpha) k^*}
    (1+x)^{k-(1-\alpha) k^*}\dx
    \\
    \leq \norm{f}_{L^{2}(0,\infty)}^{\alpha}
    \norm{f}_{L^{1}((1+x)^{k^*})}^{1-\alpha}
    \biggl(\int_{0}^{\infty}(1+x)^{\frac{2(k - (1-\alpha) k^*)}{\alpha}}\dx\biggr)^{\frac{\alpha}{2}}.
  \end{multline*}
  Since $\alpha < 2((1-\alpha) k^*- k)$, the remaining integral on the
  right-hand side can easily be computed as
  \begin{equation*}
    \biggl(\int_{0}^{\infty}(1+x)^{\frac{2(k-(1-\alpha)k^*)}{\alpha}}\dx\biggr)^{\frac{\alpha}{2}}
    =\Bigl(\frac{\alpha}{2((1-\alpha) k^*-k)-\alpha}\Bigr)^{\frac{\alpha}{2}},
  \end{equation*}
  which finishes the proof.
\end{proof}

\begin{lemma}\label{Lem:constant:kernel:L1:conv}
  Let $k^* > k \geq 2$, and let $f$ be a solution
  to~\eqref{eq:Smol:selfsim} for the constant kernel $K=2$ with total
  mass one and initial condition
  $f_0 \in L^{2}\cap L^{1}_{k^*}$, and let $G^{0}(x)=\ee^{-x}$
  be the unique stationary solution to~\eqref{eq:Smol:selfsim} with
  total mass one. There exist constants $C, \beta > 0$ such that
  \begin{equation*}
    \norm{f(t,\cdot)-G^{0}}_{L^1_k}
    \leq C \ee^{-\beta t} \qquad \text{for all $t\geq 0$.}
  \end{equation*}
  The constant $C>0$ depends only on $k$, $k^*$, $\norm{f_0}_{L^2}$,
  and $\norm{f_0}_{L^1_{k^*}}$ (and depends increasingly on the latter
  two). The constant $\beta>0$ depends only on $k$ and $k^*$.
\end{lemma}

\begin{proof}
  Due to \cite[Lemma~2.8]{EMR05} we have that $f\in L^{\infty}(0,T,L^{2})$ for all $T>0$. Thus, from Lemma~\ref{Lem:interpolation1} we know that, with an
  appropriate choice of $\alpha \in (0,1)$ (depending on $k$ and
  $k^*$),
  \begin{equation*}
    \norm{f^{0}(t,\cdot)-G^{0}}_{L^{1}_k}
    \leq
    C_1 \norm{f^{0}(t,\cdot)-G^{0}}_{L^{2}}^{\alpha}
    \,
    \norm{f^{0}(t,\cdot)-G^{0}}_{L^{1}_{k^*}}^{1-\alpha},
  \end{equation*}
  for some $C_1 > 0$ depending only on $k$ and $k^*$. Thus, Young's
  inequality (with a parameter $\delta$) for $p_{1}=1/\alpha$ and
  $p_2=1/(1-\alpha)$ yields
  \begin{equation*}
    \norm{f^{0}(t,\cdot)-G^{0}}_{L^{1}_k}
    \leq
    C_2
    \delta^{\frac{\alpha-1}{\alpha}}
    \norm{f^{0}(t,\cdot)-G^{0}}_{L^{2}}
    +
    C_1 \delta\norm{f^{0}(t,\cdot)-G^{0}}_{L^{1}_{k^*}},
  \end{equation*}
  again for $C_2 > 0$ depending only on $k$, $k^*$.
  Using \cref{Prop:self:sim:evolution:existence,Thm:L2:convergence} we
  obtain
  \begin{equation*}
    \norm{f^{0}(t,\cdot)-G^{0}}_{L^{1}_k}
    \leq
    C_3 \delta^{\frac{\alpha-1}{\alpha}}
    e^{-\frac{1}{2} t}
    +
    \delta C_2,
  \end{equation*}
  for some $C_3 > 0$ depending only on $k$, $k^*$, $\norm{f_0}_{L^2}$
  and $\norm{f_0}_{L^1_2}$, and $C_2 > 0$ depending only on $k^*$ and
  $\norm{f_0}_{L^1_{k^*}}$. Taking
  $\delta=\ee^{-\alpha t/2}$ we deduce that
  \begin{equation*}
    \norm{f^{0}(t,\cdot)-G^{0}}_{L^{1}_k}
    \leq
    C_3
    \ee^{-\frac{\alpha}{2} t}
    +
    C_2 \ee^{-\frac{\alpha}{2} t}
    = (C_2 + C_3) \ee^{-\frac{\alpha}{2} t}.
  \end{equation*}
  In sum, the constants $C_2$ and $C_3$ depend on $k$, $k^*$,
  $\norm{f_0}_{L^2}$, $\norm{f_0}_{L^1_2}$ and
  $\norm{f_0}_{L^1_{k^*}}$. Since
  $\norm{f_0}_{L^1_2} \leq \norm{f_0}_{L^1_{k^*}}$ (and $C_2$, $C_3$
  are increasing in $\norm{f_0}_{L^1_2}$), one can always modify the
  constants to have them depend only on $k$, $k^*$, $\norm{f_0}_{L^2}$,
  and $\norm{f_0}_{L^1_{k^*}}$. Taking $\beta := \alpha/2$ and
  $C := C_2 + C_3$, this shows the result.
\end{proof}

\section{Bounds and stability of self-similar profiles}
\label{sec:bound_stability}

We gather in this section some basic results on existence and bounds
which apply in particular to the self-similar profiles for the
perturbed equation. More importantly, we give some \emph{stability}
results showing that any self-similar profile $G_\eps$ with mass
one for the kernel $K_\eps$ must be close to $G_0(x) = e^{-x}$, in
distances given by suitable norms. In general, these stability results
cannot be obtained from the linearisation methods in this paper, so we
borrow them from elsewhere or prove them using different
methods. (However, linearisation methods do give some results on
\emph{local} stability of profiles, assuming we are in a certain
neighbourhood of the profile $G_0$; see Lemma
\ref{lem:local_stability}.)

\subsection{Existence of self-similar profiles}
\label{sec:exist-self-simil}

Existence of self-similar profiles for large classes of non-solvable
kernels with power-law structure was shown in
\cite{FoL05,EMR05,EsM06}, and precise results in the case of
homogeneity zero are given in \cite{NiV14}. Except for the works
mentioned in the introduction, uniqueness of scaling profiles in not known for most
coagulation kernels. However, there is a number of works providing a
priori regularity and asymptotics of self-similar solutions for small
and large cluster sizes (e.g.\@ \cite{FoL06,NiV14,MC11}). 

In our particular setting of homogeneity
zero, we cite the following result from \cite[Prop.\@ 1.1]{NiV14} which
provides existence of self-similar profiles with finite mass for the
kernels we consider (see also Remark~\ref{Rem:profile:integrated}):

\begin{proposition}[Existence of profiles]
  \label{Prop:profile:exists}
  Let $K$ be homogeneous of degree zero and let $k_{0},K_{0}>0$ and
  $\kappa\in(0,1]$ be constants such that
  $K(x,y)\leq K_{0}((x/y)^{\alpha}+(y/x)^{\alpha})$ for all
  $x,y\in (0,\infty)$ with $\alpha\in [0,1)$ and 
  $\min_{\abs{x-y}\leq\kappa(x+y)}K(x,y)\geq k_{0}$. Then there exists
  a self-similar profile $G\in C(0,\infty)\cap L^{1}(x\dx)$
  of~\eqref{eq:Smol}.
\end{proposition}

\subsection{Bounds on self-similar profiles}

In this subsection, we provide several a-priori estimates for
self-similar profiles. More precisely,
Proposition~\ref{Prop:profile:m0} provides precise upper and lower
bounds for the integral of perturbed self-similar profiles while
Proposition~\ref{Prop:L1:est:profile} states that self-similar
profiles are uniformly bounded in the $L^{1}_{k}$ norm.

\begin{proposition}\label{Prop:profile:m0}
  Let $K_{\eps}$ satisfy \cref{eq:h1,eq:h2,eq:h3} and let $G_{\eps}$ be a stationary solution of~\eqref{eq:Smol:selfsim}, i.e.\@ a self-similar profile. Then we have
  \begin{equation*}
    \frac{1}{1+\frac{\eps}{2}}\leq \int_{0}^{\infty}G_{\eps}(x)\dx\leq 1
  \end{equation*}
  for all $\eps\geq 0$.
\end{proposition}

\begin{proof}
  By assumption $G_{\eps}$ satisfies~\eqref{eq:Smol:selfsim} with left-hand side zero. Integrating this equation, we find
  \begin{multline*}
    2\int_{0}^{\infty}G_{\eps}(x)\dx+\int_{0}^{\infty}xG'_{\eps}(x)\dx+\frac{1}{2}\int_{0}^{\infty}\int_{0}^{x}(2+\eps W(x-y,y))G_{\eps}(x-y)G_{\eps}(y)\dy\dx\\
    -\int_{0}^{\infty}\int_{0}^{\infty}(2+\eps W(x,y)G_{\eps}(x)G_{\eps}(y)\dx\dy=0.
  \end{multline*}
  Denoting $m_{0}=\int_{0}^{\infty}G_{\eps}(x)\dx$, integrating by parts in the second term on the left-hand side and using Fubini's Theorem for the two double integrals, this reduces to
  \begin{multline*}
    0=2m_{0}-m_{0}+m_{0}^{2}-2m_{0}^{2}+\frac{\eps}{2}\int_{0}^{\infty}\int_{0}^{\infty}W(x,y)G_{\eps}(x)G_{\eps}(y)\dx\dy\\*
    -\eps\int_{0}^{\infty}\int_{0}^{\infty}W(x,y)G_{\eps}(x)G_{\eps}(y)\dx\dy.
  \end{multline*}
  Combining terms, we end up with
  \begin{equation*}
    m_{0}=m_{0}^{2}+\frac{\eps}{2}\int_{0}^{\infty}\int_{0}^{\infty}W(x,y)G_{\eps}(x)G_{\eps}(y)\dx\dy.
  \end{equation*}
  Due to~\eqref{eq:h2} and the non-negativity of $G_{\eps}$ we have $0\leq W(x,y)G_{\eps}(x)G_{\eps}(y)\dx\dy \leq m_{0}^{2}$ which leads to
  \begin{equation*}
    m_{0}^{2}\leq m_{0}\leq \Bigl(1+\frac{\eps}{2}\Bigr) m_{0}^{2}
  \end{equation*}
  from which the claim directly follows.
\end{proof}

Based on the previous proposition, we can also show the following statement which gives uniform boundedness of all non-negative moments for self-similar profiles. We also note that this result also follows from estimates in \cite{NiV14} but to be self-contained, we include the complete proof.

\begin{proposition}
  \label{Prop:L1:est:profile}
  Let $K_{\eps}$ satisfy \cref{eq:h1,eq:h2,eq:h3} and $k\geq 0$. Then
  there exists a uniform constant $C_{k}$ (depending only on $k$) such
  that
  \begin{equation*}
    \norm{G_{\eps}}_{L^{1}_{k}}=\int_{0}^{\infty}G_{\eps}(x)(1+x)^{k}\dx\leq C_{k}
  \end{equation*}
  for all self-similar profiles $G_{\eps}$ and all $\eps\in[0,1]$.
\end{proposition}

\begin{proof}
  We note that $(1+x)^{k}\leq C (1+x^{k})$ for some $C = C(k) >
  0$. According to Proposition~\ref{Prop:profile:m0} it thus suffices
  to show that $\int_{0}^{\infty}x^{k}G_{\eps}(x)\dx\leq C_{k}$ for
  all $k\in\N$ with $k\geq 2$.
  
  To see this, we will argue by induction. Precisely, for a fixed $k\in\N$ with $k\geq 2$ assume that the moments up to order $k-1$ are bounded uniformly, i.e.\@ $\int_{0}^{\infty}x^{\ell}G_{\eps}(x)\dx\leq C_{\ell}$ for all $\ell\leq k-1$. For $R>1$ let $\varphi_{R}$ be the linear continuation of $x\mapsto x^{k}$, i.e.\@
  \begin{equation*}
    \varphi_{R}(x)=\begin{cases}
      x^{k} & \text{if }x\leq R\\
      kR^{k-1}x-(k-1)R^{k} & \text{if }x>R.
    \end{cases}
  \end{equation*}
  Since by assumption $G_{\eps}$ has finite moments of order zero and one, we can take $\varphi_{R}$ as test function in the weak formulation of self-similar profiles which yields
  \begin{multline}\label{eq:profile:higher:moment:1}
    \int_{0}^{\infty}G_{\eps}(x)\bigl((x\varphi_{R}(x))'-2\varphi_{R}(x)\bigr)\\*
    =\biggl(\frac{1}{2}\int_{0}^{\infty}K(x,y)G_{\eps}(x)G_{\eps}(y)\bigl[\varphi_{R}(x+y)-\varphi_{R}(x)-\varphi_{R}(y)\bigr]\dy\biggr)\dx.
  \end{multline}
  It is easy to check that 
  \begin{equation*}
    \varphi_{R}(x+y)-\varphi_{R}(x)-\varphi_{R}(y)\leq \widehat{C}_{k}\bigl(x^{k-1}y+xy^{k-1}\bigr)
  \end{equation*}
  for all $x,y>0$ independent of $R$. Moreover, a direct computation yields
  \begin{equation*}
    (x\varphi_{R}(x))'-2\varphi_{R}(x)=\begin{cases}
      (k-1)x^{k} & \text{if }x\leq R\\
      (k-1)R^{k} & \text{if }x>R.
    \end{cases}.
  \end{equation*}
  Thus, we deduce from~\eqref{eq:profile:higher:moment:1} together with \cref{eq:h1,eq:h2} that
  \begin{multline*}
    (k-1)\int_{0}^{R}x^{k}G_{\eps}(x)\dx+(k-1)R^{k-1}\int_{R}^{\infty}G_{\eps}(x)\dx\\*
    \leq \frac{3}{2}\widehat{C}_{k}\int_{0}^{\infty}\int_{0}^{\infty}G_{\eps}(x)G_{\eps}(y)\bigl[x^{k-1}y+xy^{k-1}\bigr]\dx\dy.
  \end{multline*}
  Due to the non-negativity of $G_{\eps}$ we get in particular the estimate
  \begin{equation*}
    (k-1)\int_{0}^{R}x^{k}G_{\eps}(x)\dx \leq 3\widehat{C}_{k}\biggl(\int_{0}^{\infty}x^{k-1}G_{\eps}(x)\dx\biggr)\biggl(\int_{0}^{\infty}yG_{\eps}(y)\dy\biggr)
  \end{equation*}
  from which the claim follows since the right-hand side is uniformly bounded by the induction assumption.
\end{proof}

\subsection{Behaviour of profiles close to zero and a uniform \texorpdfstring{$L^2$}{L2} estimate}

The main goal of this subsection is to provide a uniform bound in $L^2$ for self-similar profiles (Proposition~\ref{Prop:L2:apriori}). The main task for this consists in deriving the behaviour of the self-similar profiles for small values of $x$ which will be done in the following sequence of lemmata. As a byproduct, we also obtain an a-priori estimate for certain negative moments, depending on the perturbation parameter $\eps$ (Lemma~\ref{Lem:profile:neg:moments}).

The first lemma provides a lower integral bound on the profiles which shows that the self-similar solutions can not concentrate at zero.

\begin{lemma}\label{Lem:profile:lower:bd:NEW}
  Let $K_{\eps}$ satisfy \cref{eq:h1,eq:h2,eq:h3}. For each $a_{*}\in (0,1)$ there exists a constant $c_{*}>0$ such that 
  \begin{equation*}
    \int_{a_{*}}^{\infty}G_{\eps}(x)\dx\geq c_{*}
  \end{equation*}
  for all self-similar profiles $G_{\eps}$ with $\int_{0}^{\infty}xG_{\eps}(x)\dx=1$.
\end{lemma}

\begin{proof}
  Splitting the integral we find together with Cauchy's inequality and Proposition~\ref{Prop:profile:m0} that
  \begin{multline*}
    1=\int_{0}^{\infty}xG_{\eps}(x)\dx=\int_{0}^{a_{*}}xG_{\eps}(x)\dx+\int_{a_{*}}^{\infty}xG_{\eps}(x)\dx\leq a_{*}\int_{0}^{a_{*}}G_{\eps}(x)\dx+\int_{a_{*}}^{\infty}xG_{\eps}(x)\dx\\*
    \leq a_{*}+\biggl(\int_{a_{*}}^{\infty}x^2 G_{\eps}(x)\dx\biggr)^{1/2}\biggl(\int_{a_{*}}^{\infty}G_{\eps}(x)\dx\biggr)^{1/2}.
  \end{multline*}
  Together with Proposition~\ref{Prop:L1:est:profile} we find
  \begin{equation*}
    1-a_{*}\leq C_{2}^{1/2}\biggl(\int_{a_{*}}^{\infty}G_{\eps}(x)\dx\biggr)^{1/2}
  \end{equation*}
  where $C_{2}$ is a uniform bound on the second moment which is provided by Proposition~\ref{Prop:L1:est:profile}. Thus, since $a_{*}\in (0,1)$ we conclude
  \begin{equation*}
    \int_{a_{*}}^{\infty}G_{\eps}(x)\dx\geq \frac{(1-a_{*})^2}{C_{2}}
  \end{equation*}
  and the claim follows with $c_{*}=(1-a_{*})^2/C_{2}$.
\end{proof}

The next statement gives an estimate on the primitive for self-similar profiles close to zero.

\begin{lemma}\label{Lem:beh:zero:weak}
  Let $K_{\eps}$ satisfy \cref{eq:h1,eq:h2,eq:h3}. There exist constants $C_*>0$ and $\eps_*\in (0,1)$ such that for $\eps\in(0,\eps_{*})$ each self-similar profile $G_{\eps}$ satisfies
  \begin{equation*}
    \int_{0}^{x}G_{\eps}(y)\dy\leq C_{*}x^{1-\frac{2\eps}{2+\eps}} \qquad \text{for all }x\leq 1.
  \end{equation*}
  In particular, this implies $\int_{0}^{x}G_{\eps}(y)\dy\leq C_{*}x^{1-\eps}$.
\end{lemma}

\begin{proof}
  To simplify the notation, we denote $P(x)\vcc=\int_{0}^{x}G_{\eps}(y)\dy$. We first note that it suffices to prove the claim for $x\leq a_{*}$ with $a_{*}\in (0,1)$ fixed. In fact, for $x>a_{*}$ we obtain by means of Proposition~\ref{Prop:bound:m0} that $P(x)\leq 1\leq x/a_{*}$ and thus the claimed estimate holds with $C_{*}=1/a_{*}$.
  
  To prove the statement for $x\leq a_{*}$, we integrate the stationary version of~\eqref{eq:Smol:selfsim} for $K=K_{\eps}$ over $[0,x]$ to obtain
  \begin{multline*}
    2P(x)+\int_{0}^{x}zG'_{\eps}(z)\dz+\frac{1}{2}\int_{0}^{x}\int_{0}^{z}K_{\eps}(z-y,y)G_{\eps}(z-y)G_{\eps}(y)\dy\dz\\
    -\int_{0}^{x}\int_{0}^{\infty}K_{\eps}(z,y)G_{\eps}(y)G_{\eps}(z)\dy\dz=0.
  \end{multline*}
  Integration by parts in the first integral and applying Fubini's theorem together with the change of variables $z\mapsto z+y$ we find
  \begin{multline*}
    2P(x)+xP'(x)-P(x)+\frac{1}{2}\int_{0}^{x}\int_{0}^{x-y}K_{\eps}(z,y)G_{\eps}(y)G_{\eps}(z)\dz\dy\\*
    -\int_{0}^{x}\int_{0}^{\infty}K_{\eps}(z,y)G_{\eps}(y)G_{\eps}(z)\dy\dz=0.
  \end{multline*}
  Inserting $K_{\eps}=2+\eps W$ and summarising, this simplifies to
  \begin{multline*}
    (1-2P(\infty))P(x)+xP'(x)+(P\ast G_{\eps})(x)\\*
    +\eps\biggl(\frac{1}{2}\int_{0}^{x}\int_{0}^{x-y}W(y,z)G_{\eps}(y)G_{\eps}(z)\dz\dy-\int_{0}^{x}\int_{0}^{\infty}W(y,z)G_{\eps}(y)G_{\eps}(z)\dy\dz\biggr)=0.
  \end{multline*}
  We estimate the left-hand side from above by noting that the expression in parentheses is non-positive since the domain of integration for the negative term is larger than the one for the positive (which in addition has a factor of $1/2$). Moreover, by monotonicity we have $(P\ast G_{\eps})(x)\leq P^2(x)$. Together this implies
  \begin{equation*}
    0\leq (1-2P(\infty))P(x)+xP'(x)+P^2(x).
  \end{equation*}
  Denoting $\alpha\vcc=2P(\infty)-1=2\int_{0}^{\infty}G_{\eps}(x)\dx-1$ this can be equivalently written as
  \begin{equation*}
    -\alpha P(x)+xP'(x)\geq -P^2(x).
  \end{equation*}
  This differential inequality can be solved explicitly. In fact, using the integration factor $x^{-\alpha-1}$ we get
  \begin{equation*}
    \frac{\dd}{\dx}\bigl(x^{-\alpha}P(x)\bigr)=x^{-\alpha-1}(-\alpha P(x)+xP'(x))\geq -x^{-\alpha-1}P^2(x)=-x^{\alpha-1}\bigl(x^{-\alpha}P(x)\bigr)^2.
  \end{equation*}
  We note that $P$ is monotonically non-decreasing. Thus, if $P(a_{*})=0$ the claim is trivial. We therefore assume $P(a_{*})>0$ which allows to rewrite on an interval $(a,a_{*})$ (note that we only have to consider the region where $P$ is non-zero):
  \begin{equation*}
    -\frac{\dd}{\dx}\bigl(x^{-\alpha}P(x)\bigr)^{-1}=\frac{\dd}{\dx}\bigl(x^{-\alpha}P(x)\bigr)\bigl(x^{-\alpha}P(x)\bigr)^{-2}\geq -x^{\alpha-1}.
  \end{equation*}
  Integrating this inequality over $(x,a_{*})$ we obtain
  \begin{equation*}
    -\frac{1}{P(a_{*})}+\frac{x^{\alpha}}{P(x)}\geq -\frac{1}{\alpha}(a_{*}^{\alpha}-x^{\alpha})
  \end{equation*}
  or equivalently
  \begin{equation}\label{eq:profile:asympt:zero:1}
    \frac{x^{\alpha}}{P(x)}\geq \frac{1}{P(a_{*})}-\frac{a_{*}^{\alpha}}{\alpha}+\frac{x^{\alpha}}{\alpha}.
  \end{equation}
  The definitions of $\alpha$ and $P$ imply 
  \begin{equation*}
    \frac{1}{P(a_{*})}-\frac{a_{*}^{\alpha}}{\alpha}=\frac{1}{\alpha P(a_{*})}\bigl(\alpha-a_{*}^{\alpha}P(a_{*})\bigr)=\frac{1}{\alpha P(a_{*})}\biggl(2\int_{0}^{\infty}G_{\eps}(x)\dx-1-a_{*}^{\alpha}\int_{0}^{a_{*}}G_{\eps}(x)\dx\biggr).
  \end{equation*}
  Since $\alpha> 0$ due to Proposition~\ref{Prop:bound:m0} and $a_{*}\in (0,1)$ we have $a_{*}^{\alpha}\leq 1$ which yields together with \cref{Lem:profile:lower:bd:NEW,Prop:profile:m0} that for $\eps$ sufficiently small, we have
  \begin{multline}\label{eq:profile:asympt:zero:2}
    \frac{1}{P(a_{*})}-\frac{a_{*}^{\alpha}}{\alpha}\geq \frac{1}{\alpha P(a_{*})}\biggl(2\int_{0}^{\infty}G_{\eps}(x)\dx-1-\int_{0}^{a_{*}}G_{\eps}(x)\dx\biggr)\\*
    =\frac{1}{\alpha P(a_{*})}\biggl(\int_{a_{*}}^{\infty}G_{\eps}(x)\dx+\int_{0}^{\infty}G_{\eps}(x)\dx-1\biggr)\\*
    \geq \frac{1}{\alpha P(a_{*})}\Bigl(c_{*}+\frac{1}{1+\frac{\eps}{2}}-1\Bigr)=\frac{1}{\alpha P(a_{*})}\Bigl(c_{*}-\frac{\eps}{2+\eps}\Bigr).
  \end{multline}
  Thus, if $\eps$ is small enough, the right-hand side is strictly positive (note that $\alpha$ is strictly positive due to Proposition~\ref{Prop:profile:m0}). With this, we deduce from \cref{eq:profile:asympt:zero:1,eq:profile:asympt:zero:2} that
  \begin{equation*}
    P(x)\leq \frac{x^{\alpha}}{\frac{1}{P(a_{*})}-\frac{a_{*}^{\alpha}}{\alpha}+\frac{x^{\alpha}}{\alpha}}\leq \frac{\alpha P(a_{*})}{c_{*}-\frac{\eps}{2+\eps}}x^{\alpha}=\frac{1}{c_{*}-\frac{\eps}{2+\eps}}\biggl(2\int_{0}^{\infty}G_{\eps}(y)\dy-1\biggr)\int_{0}^{a_{*}}G_{\eps}(y)\dy x^{\alpha}.
  \end{equation*}
  Together with Proposition~\ref{Prop:profile:m0} and the non-negativity of $G_{\eps}$ the right-hand side can be further estimated to get
  \begin{equation*}
    P(x)\leq \frac{1}{c_{*}-\frac{\eps}{2+\eps}}x^{\alpha}.
  \end{equation*}
  Finally, we recall again Proposition~\ref{Prop:profile:m0} to deduce $\alpha=2\int_{0}^{\infty}G_{\eps}(x)\dx-1\geq 2/(1+\eps/2)-1\geq (2-\eps)/(2+\eps)$ from which the claimed estimate follows.
\end{proof}

Based on the preparation above, we can now provide a pointwise estimate on the behaviour of self-similar profiles close to zero.

\begin{lemma}\label{Lem:profile:zero:pointwise}
  Let $K_{\eps}$ satisfy \cref{eq:h1,eq:h2,eq:h3}. There exist constants $C_{*}>0$ and $\eps_{*}>0$ such that each self-similar profiles $G_{\eps}$ satisfies
  \begin{equation*}
    G_{\eps}(x)\leq C_{*}x^{-\frac{2\eps}{2+\eps}}\leq C_{*}x^{-\eps} \qquad \text{for almost all }x\leq 1
  \end{equation*}
  if $\eps\leq \eps_{*}$.
\end{lemma}

\begin{proof}
  We recall from Remark~\ref{Rem:profile:integrated} that $G_{\eps}$ satisfies the equation
  \begin{equation*}
    G_{\eps}(x)=\frac{1}{x^2}\int_{0}^{x}\int_{x-y}^{\infty}yK_{\eps}(y,z)G_{\eps}(y)G_{\eps}(z)\dz\dy.
  \end{equation*}
  The assumptions \cref{eq:h1,eq:h2} together with $\eps\leq 1$ imply $K_{\eps}\leq 3$ which together with the non-negativity of $G_{\eps}$ yields
  \begin{equation*}
    G_{\eps}(x)\leq \frac{3}{x^2}\int_{0}^{x}\int_{x-y}^{\infty}yG_{\eps}(y)G_{\eps}(z)\dz\dy\leq \frac{3}{x}\int_{0}^{x}\int_{0}^{\infty}G_{\eps}(y)G_{\eps}(z)\dz\dy.
  \end{equation*}
  Together with \cref{Lem:profile:lower:bd:NEW,Prop:profile:m0} we thus conclude
  \begin{equation*}
    G_{\eps}(x)\leq \frac{3}{x}\int_{0}^{x}G_{\eps}(y)\dy\leq 3C_{*}x^{-\frac{2\eps}{2+\eps}}.
  \end{equation*}
\end{proof}

The next lemma gives uniform estimates also for certain negative moments.

\begin{lemma}\label{Lem:profile:neg:moments}
  Let $\eps_{*}\in(0,1)$ and let $K_{\eps}$ satisfy \cref{eq:h1,eq:h2,eq:h3} with $0\leq \eps\leq\eps_{*}$. For each $\alpha\in (\eps_{*}-1,\infty)$ there exists a constant $C_{\alpha}$ such that each self-similar profile $G_{\eps}$ satisfies
  \begin{equation*}
    \int_{0}^{\infty}x^{\alpha}G_{\eps}(x)\dx\leq C_{\alpha}.
  \end{equation*}
\end{lemma}

\begin{proof}
  The statement is a direct consequence of \cref{Prop:L1:est:profile,Lem:profile:zero:pointwise}.
\end{proof}

The preparation above now allows us to obtain uniform estimates on the $L^2$ norm of self-similar profiles.

\begin{proposition}\label{Prop:L2:apriori}
  Let $K_{\eps}$ satisfy \cref{eq:h1,eq:h2,eq:h3}. There exist constants $C_{*}>0$ and $\eps_{*}>0$ such that each self-similar profiles $G_{\eps}$ satisfies
  \begin{equation*}
    \norm{G_{\eps}}_{L^{2}(1)}=\biggl(\int_{0}^{\infty}\abs*{G_{\eps}(x)}^2\dx\biggr)^{1/2}\leq C_{*}
  \end{equation*}
  if $\eps\leq \eps_{*}$.
\end{proposition}

\begin{proof}
  We recall from the proof of Lemma~\ref{Lem:profile:zero:pointwise} that 
  \begin{equation*}
    G_{\eps}(x)=\frac{1}{x^2}\int_{0}^{x}\int_{x-y}^{\infty}yK_{\eps}(y,z)G_{\eps}(y)G_{\eps}(z)\dz\dy.
  \end{equation*}
  Thus, multiplying by $G_{\eps}$ and integrating, together with Fubini's Theorem, we deduce
  \begin{equation*}
    \norm{G_{\eps}}_{L^{2}(1)}^{2}=\int_{0}^{\infty}\abs{G_{\eps}(x)}^{2}\dx=\int_{0}^{\infty}\int_{0}^{\infty}yK_{\eps}(y,z)G_{\eps}(y)G_{\eps}(z)\int_{y}^{y+z}\frac{G_{\eps}(x)}{x^2}\dx\dz\dy.
  \end{equation*}
  Due to \cref{eq:h1,eq:h2} and $\eps\leq 1$ we find
  \begin{equation*}
    \norm{G_{\eps}}_{L^{2}(1)}^{2} \leq 3\int_{0}^{\infty}\int_{0}^{\infty}yG_{\eps}(y)G_{\eps}(z)\int_{y}^{\infty}\frac{G_{\eps}(x)}{x^2}\dx\dz\dy.
  \end{equation*}
  Next, we note that \cref{Prop:L1:est:profile,Lem:profile:zero:pointwise} directly imply that $\int_{y}^{\infty}G_{\eps}(x)/x^2\dx\leq C y^{-1-\eps}$ for all $y>0$ (note that we could obtain a much better decay for $y>1$). Using this, we deduce
  \begin{equation*}
    \norm{G_{\eps}}_{L^{2}(1)}^{2} \leq 3\int_{0}^{\infty}y^{-\eps}G_{\eps}(y)\dy\int_{0}^{\infty}G_{\eps}(z)\dz.
  \end{equation*}
  The claim then follows from \cref{Prop:L1:est:profile,Lem:profile:neg:moments} if $\eps<1/2$.
\end{proof}

\subsection{Stability of profiles}
\label{sec:stability-profiles}

Regarding stability, the following statement is a particular case of
\cite[Thm.\@ 2.4]{Thr19} for bounded perturbations $W$.

\begin{proposition}[Stability of profiles]
  \label{Prop:profile_stability}
  Let $W$ be a bounded kernel satisfying \eqref{eq:h2} and
  \eqref{eq:h3}.  For $\eps \geq 0$, denote $K_\eps := 2 + \eps
  W$. For any $k \geq 0$ there exists a function
  $\delta = \delta(\eps)$ depending only $k$ and $\eps$, with
  $\delta(\eps) \to 0$ as $\eps\to 0$ such that any self-similar
  profile $G_\eps$ with mass $1$ of Smoluchowski's equation with
  kernel $K_\eps$ satisfies
  \begin{equation*}
    \| G_\eps - G_0 \|_{L^1_k} \leq \delta(\eps).
  \end{equation*}
\end{proposition}

This is a basic property we need in order to complete the proofs of
our main results in Sections \ref{sec:local} and
\ref{sec:convergence}. Its only drawback is that the stability rate
$\delta = \delta(\eps)$ was obtained in \cite{Thr19} via
compactness arguments, and hence one cannot give any constructive
estimate on it. As a consequence, using Proposition
\ref{Prop:profile_stability} as given, the constants in our main result
in Theorem \ref{thm:main-intro} would become non-constructive: we
would not be able to estimate the size of $\eps_1$ or $\eps_3$,
even if we know there must be one satisfying the statement.

In order to improve this situation we give an alternative way to show
Proposition \ref{Prop:profile_stability}, which yields an explicit estimate
on the rate $\delta(\eps)$.

The main idea to obtain constructive estimates is to use available
quantitative information on the asymptotic behaviour of solutions to
the self-similar equation
$\partial_t f = C_0(f,f) + 2f + x \partial_x f$ with constant
coefficients. In a broad sketch, if we know that
\begin{enumerate}
\item the equation with constant coefficients relaxes to equilibrium,
  with explicit rates, in a certain norm,
\item the dynamics of solutions depends continuously on the
  perturbation $\eps$, in the same norm,
\item and the norm of any profile $G_\eps$ is bounded by a
  uniform constant,
\end{enumerate}
then we can conclude that any profile $G_\eps$ must be close to $G_0$
for small $\eps$, in the same norm we are considering. Point 2 seems
to be the least problematic of the three, and we will use our
Lemma~\ref{Lem:L1:evolution:close}. Let us see what is available
regarding point 1. Since \cite{Menon2004Approach} it is known that
solutions in the constant coefficients case converge to equilibrium,
and a quantitative estimate of the rate at which this happens was
given in \cite{CMM10}, in several norms including $L^2$ and weighted
$L^2$ norms. A clean statement with explicit constants was then given
in \cite[Theorem 1.1]{doi:10.1137/090759707}, for
$\norm{\cdot}_{W^{-1,\infty}}$. Since we want to show that
$\|G_0 - G_\eps\|_{L^1_k}$ is small, we are forced to use the $L^2$ norm
convergence result in \cite{CMM10} to fulfill point 1 since a simple
interpolation then allows us to control the $L^1_k$ norm. We have
stated this result in our \cref{Lem:constant:kernel:L1:conv}. Notice
that we are also restricted by point 3, since uniform estimates of
profiles are available in $L^1_k$, but not for example in $L^\infty$
or $W^{1, 1}$. For point 3, we use uniform estimates of profiles in
$L^2$ given in \cref{Prop:L2:apriori}, which as far as we know were
not available elsewhere.

These ideas give us a proof of \cref{Prop:profile_stability} with an
explicit $\delta(\eps)$:

\begin{proof}[Proof \cref{Prop:profile_stability}]
  Take any solution $f^{\eps}$ to~\eqref{eq:Smol:selfsim} with kernel
  $K=2+\eps W$, and any solution $f^0$ to~\eqref{eq:Smol:selfsim} with
  constant kernel
  $K=2$. We have
  \begin{equation}
    \label{eq:pstab1}
    \norm{f^{\eps}(t,\cdot)-G^{0}}_{L^1_k}
    \leq
    \norm{f^{\eps}(t,\cdot)- f^0(t,\cdot)}_{L^1_k}
    +
    \norm{f^0(t,\cdot) - G^0}_{L^1_k}.
  \end{equation}
  Let $G^\eps$ be any self-similar profile with mass one for the
  kernel $K_\eps$, and choose the initial condition for both $f^\eps$
  and $f^0$ to be equal to $G^\eps$. In particular, $f^\eps$ is then
  equal to the constant-in-time profile $G^\eps$. From
  \cref{Lem:L1:evolution:close},
  \begin{equation*}
    \norm{f^{\eps}(t, \cdot)- f^0(t,\cdot)}_{L^1_k}
    =
    \norm{G^{\eps}- f^0(t,\cdot)}_{L^1_k}
    \leq
    \eps C_1 \ee^{C_1 t}.
  \end{equation*}
  for some $C_1 > 0$ depending only on $\norm{G^\eps}_{L^1_k}$, in an
  increasing way. Since we know from Lemma \ref{Prop:L1:est:profile}
  that $\norm{G^\eps}_{L^1_k}$ is uniformly bounded by a constant
  $C_k$ depending only on $k$, we conclude that the constant $C_1$ can
  be chosen to depend only on $k$ as well. On the other hand,
  \cref{Lem:constant:kernel:L1:conv} shows that
  \begin{equation*}
    \norm{f^0(t,\cdot) - G^0}_{L^1_k}
    \leq
    C_2 \ee^{-\beta t}
  \end{equation*}
  for some constant $C_2 > 0$ depending on $k$, $\norm{G^\eps}_{L^2}$
  and $\norm{G^\eps}_{L^1_{k+1}}$ (for example; any moment larger than
  $k$ will do, not necessarily $k+1$). The constant $\beta > 0$
  depends only on $k$. In a similar way as before, since we know from
  \cref{Prop:L1:est:profile,Prop:L2:apriori} that these norms of
  $G^\eps$ are uniformly bounded by constants that depend only on $k$
  we conclude that the constant $C_2$ can be chosen to
  depend only on $k$. Using our last two estimates in
  \eqref{eq:pstab1},
  \begin{equation*}
    \norm{G^\eps - G^{0}}_{L^1_k}
    =
    \norm{f^{\eps}(\cdot,t)-G^{0}}_{L^1_k}
    \leq
    \eps C_1 \ee^{C_1 t} + C_2 \ee^{-\beta t},
    \qquad
    t \geq 0.
  \end{equation*}
  for each solution $f^{\eps}$ to~\eqref{eq:Smol:selfsim} with
  $K=2+\eps W$. Choosing $t=\log(C_{2}\beta/(\eps C_{1}^{2}))/(C_{1}+\beta)$ the claim follows with $\delta(\eps)=\eps^{\beta/(C_{1}+\beta)}$.

  \begin{remark}
    In Section~\ref{sec:local_stability} we will give a further improvement of Proposition~\ref{Prop:profile_stability}, i.e.\@ we will show that actually $\delta(\eps)=\mathcal{O}(\eps)$ as $\eps\to 0$.
  \end{remark}
\end{proof}

\section{The linearised operator for the constant kernel and semigroup theory}
\label{sec:lin_operator}

In this section we will introduce the linearised coagulation operator $\L_{0}$ in self-similar variables. Precisely, if we linearise the stationary version of~\eqref{eq:Smol:selfsim}
around the profile $G_{0}(x)=\ee^{-x}$ we get 
\begin{equation}
  \label{eq:linearised}
  \L_{0}[h] =
  x\del_{x}h+2(G_{0}\ast h) - 2G_{0}\int_{0}^{\infty}h(y)\dy.
\end{equation}
Since this expression contains a derivative with respect to $x$ which
is not defined in the spaces we consider, $\L_{0}$ has to be
understood as an unbounded operator. However, $\L_0$ is obviously
well-defined on $C_{c}^{\infty}(0,\infty)$ by the following equivalent
representation formulas:

\begin{lemma}\label{Lem:rewriting:L}
  On the space $C_{c}^{\infty}(0,\infty)$ the operator $\L_{0}$ as given in~\eqref{eq:linearised} can be equivalently written as
  \begin{align}
    \L_{0}[h](x) &=
                   x h'(x) + 2 e^{-x} \int_0^x h(y) (e^{y} - 1) \dy
                   - 2e^{-x} \int_{x}^{\infty}h(y)\dy  \label{eq:linearised2}
                   \shortintertext{and}
                   \L_{0}[h](x)&=xh'(x)-2H(x)+2\int_{0}^{x}H(y)\ee^{-(x-y)}\dy \label{eq:linearised3}
  \end{align}
  where $H(x)=\int_{x}^{\infty}h(z)\dz$.
\end{lemma}

\begin{proof}
  The expression~\eqref{eq:linearised2} is obvious by splitting and
  re-combining the two integrals. For~\eqref{eq:linearised3} we
  rewrite $h(y)=-\del_{y}(\int_{y}^{\infty}h(z)\dz)$ in the
  convolution expression of~\eqref{eq:linearised} and integrate by
  parts which yields
  \begin{multline*}
    \L_{0}[h](x)=xh'(x)+2\int_{0}^{x}-\del_{y}\int_{y}^{\infty}h(z)\dz\ee^{-(x-y)}\dy-2\ee^{-x}\int_{0}^{\infty}h(y)\dy\\*
    =xh'(x)-2\int_{x}^{\infty}h(z)\dz+2\int_{0}^{\infty}h(z)\dz\ee^{-x}+2\int_{0}^{x}\int_{y}^{\infty}h(z)\dz\ee^{-(x-y)}\dy-2\ee^{-x}\int_{0}^{\infty}h(y)\dy\\*
    =-2H(x)+2\int_{0}^{x}H(y)\ee^{-(x-y)}\dy.
  \end{multline*}
\end{proof}

The main point of this section is to prove that $\L_0$, defined on a
suitable domain, generates a strongly continuous semigroup in the
spaces in which we will be working later:
\begin{theorem}
  \label{thm:semigroups-defined}
  There are semigroups defined on each of the spaces
  $L^{1}((1+x)^{k})$ for all $k\geq 0$, $L^{2}(\ee^{\mu x})$ and
  $H^{-1}(\ee^{\mu x})$ for all $\mu\geq 0$, such that
  \begin{enumerate}
  \item their generators are all defined on $\mathcal{C}_{c}^\infty$,
    and they are equal to $\L_0$ on $\mathcal{C}_{c}^\infty$,
  \item and $\mathcal{C}_{c}^\infty$ is a core for their generators (see Definition~\ref{Def:core} below).
  \end{enumerate}
  These semigroups can all be restricted to the corresponding spaces
  with mass zero, i.e.\@ their intersections with
  $\{\int_{0}^{\infty}xf(x)\dx=0\}$
\end{theorem}
When talking about $\L_0$ on any of these spaces, it is implicit that
we mean the generator of the corresponding semigroup (equivalently,
the closure of $\L_0$, defined on $\mathcal{C}_c^\infty$, in the
corresponding norm). 

The rest of this section is devoted to the proof of
\cref{thm:semigroups-defined}. To simplify working with unbounded
operators and in particular with the corresponding domains in the
following we introduce the notion of a core which usually allows to
restrict to dense subsets instead of the full domain of the
operator. The following definition is taken from \cite[Ch. I,
Definition 1.6]{EnN00}:

\begin{definition}[Core]\label{Def:core}
  For an unbounded operator $U\colon D(U)\subset X\to X$ a subspace
  $S\subset D(U)$ is denoted \emph{core} of $U$ if $S$ is dense in
  $D(U)$ for the graph norm $\norm{h}_{U}=\norm{h}+\norm{Uh}$.
\end{definition}

The next lemma is an extension of the Bounded  Perturbation Theorem stating that the latter also preserves the core of an unbounded operator. Since this result seems not to be proved in \cite{EnN00}, we present the short proof for completeness.

\begin{lemma}\label{Lem:ext:core:general}
  Let $U\colon D(U)\subset X\to X$ be the generator of a strongly continuous semigroup $\ee^{U t}$ with core $S\subset D(U)$. Furthermore, let $V\colon X\to X$ be bounded. Then $U+V$ is the generator of a strongly continuous semigroup $\ee^{(U+V)t}$ on $X$ with domain $D(U)$ and core $S$.
\end{lemma}

\begin{proof}
  According to the Bounded Perturbation Theorem, the operator $U+V$ is the generator of a strongly continuous semigroup on $X$ with domain $D(U)$. To see that $S$ is still a core, it suffices to prove that the graph norms of $U$ and $U+V$, i.e.\@ $\norm{\cdot}_{U}$ and $\norm{\cdot}_{U+V}$ are equivalent. For this, we fix $\kappa>1+\norm{V}$ which yields
  \begin{multline*}
    \norm{h}_{U+V}=\norm{h}+\norm{(U+V)h}\geq \norm{h}+\frac{1}{\kappa}\norm{(U+V)h}\geq \norm{h}+\frac{1}{\kappa}\norm{Uh}-\frac{1}{\kappa}\norm{Vh}\\*
    \geq \Bigl(1-\frac{\norm{V}}{\kappa}\Bigr)\norm{h}+\frac{1}{\kappa}\norm{Uh}\geq \frac{1}{\kappa}\bigl(\norm{h}+\norm{Uh}\bigr)=\frac{1}{\kappa}\norm{h}_{U}.
  \end{multline*}
  Conversely, we find
  \begin{multline*}
    \norm{h}_{U+V}=\norm{h}+\norm{(U+V)h}\leq \norm{h}+\norm{Uh}+\norm{Vh}\leq \bigl(1+\norm{V}\bigr)\norm{h}+\norm{Uh}\\*
    \leq \bigl(1+\norm{V}\bigr)\bigl(\norm{h}+\norm{Uh}\bigr)\leq \bigl(1+\norm{V}\bigr)\norm{h}_{U}.
  \end{multline*}
  This finishes the proof.
\end{proof}

The following remark states more precisely how the action of linear operators on the space $H^{-1}(\ee^{\mu x})$ will be understood in the following.

\begin{remark}[Definition of linear operators on $H^{-1}$]\label{Rem:lin:op:H1}
  Since $H^{-1}(\ee^{\mu x})$ is a rather weak space, it seems most
  appropriate to define linear operators on it via a density
  argument. Precisely, due to the definition of $H^{-1}(\ee^{\mu x})$,
  the elements of this space are represented by equivalence classes of
  Cauchy sequences with respect to
  $\norm{\cdot}_{H^{-1}(\ee^{\mu x})}$. In particular, for each class,
  we can always find a representative sequence which is contained in
  $C_{c}^{\infty}(0,\infty)$. The approach then consists in defining a
  given linear operator on this space and extend it again by
  completion with respect to the norm on $H^{-1}(\ee^{\mu x})$. This
  procedure of course requires that the operator defined this way in
  fact maps to $H^{-1}(\ee^{\mu x})$ and that the definition is
  independent of the specific choice of a sequence. However, both
  properties are obviously satisfied if the operator $U$ to be defined
  this way satisfies
  \begin{equation*}
    \norm{Uh}_{H^{-1}(\ee^{\mu x})}\leq C \norm{h}_{H^{-1}(\ee^{\mu x})} \qquad \text{for all }h\in C_{c}^{\infty}(0,\infty),
  \end{equation*}
  i.e.\@ the restriction of $U$ to $C_{c}^{\infty}(0,\infty)$ is bounded. For the operators considered in this work, the latter property will typically be satisfied and in this case, we implicitly use the construction described before.
\end{remark}

The following lemma provides that $h\mapsto xh'(x)$, which appears in the linearised coagulation operator, is the generator of a strongly continuous semigroup in the spaces $L^1_{k}$, $L^2(\ee^{\mu x})$ and $H^{-1}(\ee^{\mu x})$.

\begin{lemma}\label{Lem:diff:sg:NEW}
  The family of operators $(T_{t})_{t\geq 0}$ given by the formula
  $(T_{t}h)(x)=h(x\ee^{t})$ defines a strongly continuous semigroup on
  $L^{1}((1+x)^{k})$ for all $k\geq 0$ as well as on
  $L^{2}(\ee^{\mu x})$ and $H^{-1}(\ee^{\mu x})$ for all $\mu\geq
  0$. Moreover, this defines also a semigroup on the corresponding
  spaces restricted to total mass equal to zero, i.e.\@
  $L^{1}((1+x)^{k})\cap\{\int_{0}^{\infty}xf(x)\dx=0\}$ for all
  $k\geq 0$ as well as
  $L^{2}(\ee^{\mu x})\cap\{\int_{0}^{\infty}xf(x)\dx=0\}$ and
  $H^{-1}(\ee^{\mu x})\cap\{\int_{0}^{\infty}xf(x)\dx=0\}$ for
  $\mu>0$. In all cases the generator $\B_{1}$ is given by
  $\B_{1}h=x\frac{\dd}{\dx}\colon h\mapsto xh'(x)$ (while, by abuse of
  notation, we use the same notation for the generator on different
  spaces) and the space $C_{c}^{\infty}(0,\infty)$ or
  $C_{c}^{\infty}(0,\infty)\cap\{\int_{0}^{\infty}xf(x)\dx=0\}$
  respectively is a core.
  
  Moreover, we have
  \begin{equation*}
    \begin{aligned}
      \norm{T_{t} h}_{L^{1}_{k}}&\leq \norm{h}_{L^{1}_{k}}\ee^{-t} && \text{for all }h\in L^{1}_{k}\\
      \norm{T_{t} h}_{L^{2}(\ee^{\mu x})}&\leq \norm{h}_{L^{2}(\ee^{\mu x})}\ee^{-\frac{1}{2}t} && \text{for all }h\in L^{2}(\ee^{\mu x})\\
      \norm{T_{t} h}_{H^{-1}(\ee^{\mu x})}&\leq \norm{h}_{H^{-1}(\ee^{\mu x})}\ee^{-\frac{3}{2}t} && \text{for all }h\in H^{-1}(\ee^{\mu x})
    \end{aligned}
  \end{equation*}
  and all $t\geq 0$.
\end{lemma}

Since the semigroup is explicitly given, the proof is straightforward
but for the sake of completeness, we include it in
\cref{Sec:proof:sg}.

\medskip

\begin{proof}[Proof of \cref{thm:semigroups-defined}]
  We just apply the bounded perturbation theorem to the semigroups
  given by \cref{Lem:diff:sg:NEW}, since the reader can check that the
  remaining terms in the definition of $\L_0$ are bounded operators
  (in the $L^1$ and $L^2$ cases one can choose expression
  \eqref{eq:linearised2} for this; in the $H^{-1}$ case one can choose
  expression \eqref{eq:linearised3}).
\end{proof}

\section{Tools on the spectral gap of linear operators}\label{Sec:tools:sg}

The following technique allows us to obtain spectral gaps in different
spaces, once a spectral gap in some space has been proved. These ideas
stem from classical perturbation theory of linear operators, with
constructive estimates given in \cite{Mouhot2006Rate} and a general
theory developed in \cite{GMM}. The simple approach described here was
already used in \cite{Canizo2013Exponential}, and we describe it
below.

\subsection{Spectral gap}

We often refer to our estimates on the decay of several semigroups as
``spectral gap estimates''. The main decay property that we are
interested in is more precisely called \emph{hypodissipativity}:

\begin{definition}[Hypodissipative semigroup]
  \label{Def:sg}
  Let $X$ be a Banach space and
  $\mathcal{A}\colon D(\mathcal{A})\subset X\to X$ the generator of a
  strongly continuous semigroup $\ee^{\mathcal{A}t}$. We say that
  $\mathcal{A}$ is \emph{hypodissipative} (or that the semigroup
  $(e^{\mathcal{A} t})_{t \geq 0}$ is hypodissipative) if there exist
  constants $C \geq 1$, $\lambda > 0$ such that
  \begin{equation*}
    \norm{\ee^{\mathcal{A}t}h}_{X}\leq C\norm{h}_{X}\ee^{-\lambda t}
  \end{equation*}
  for all $h\in X$.
\end{definition}

An operator $\mathcal{A}$ is usually called \emph{dissipative} if the
above definition holds with $C = 1$. If $\mathcal{A}$ is the generator
of a strongly continuous semigroup on a Banach space $X$, we say
$\mathcal{A}$ has a spectral gap if its kernel is nonzero (i.e., there
are \emph{equilibria} of the evolution), and the corresponding
semigroup is hypodissipative when restricted to a suitable subspace of
$X$ which does not contain the kernel of $\mathcal{A}$ (usually the
subspace perpendicular to the kernel in a suitable scalar
product). Since we do not define this ``suitable subspace'' in
general, every time we mention a spectral gap result it should be
clear that we always refer to a specific decay property of the
corresponding semigroup.

Of course, this is intimately related to the property that
the spectrum of $\mathcal{A}$ consists of $0$, plus an additional set
contained in $\{z \in \mathbb{C} \mid \Re(z) \leq -\lambda \}$, but
there is not a simple equivalence without further decay properties of
the resolvent (by the Hille-Yosida theorem). This is why we prefer to
work only with estimates on the decay of the associated semigroups.

\subsection{Restriction of the spectral gap}

We state a result similar to \cite[Theorem 3.1]{Canizo2013Exponential}
or \cite[Theorem 1.1]{Mischler2016}, dealing with restriction of the
spectral gap of a linear operator instead of extension. Strictly
speaking, the results below allow us to transfer the hypodissipativity
property between semigroups. In order to use them to transfer a
spectral gap property, we will later apply them to the subspaces
perpendicular to the equilibrium in a suitable sense.

\begin{theorem}
  \label{thm:restriction:NEW}
  Consider two Banach spaces $\mathcal{Y} \subseteq \mathcal{Z}$ with corresponding norms $\mathcal{Y}$ and $\mathcal{Z}$ such that $\|h\|_\mathcal{Z} \leq C_{\mathcal{Y}} \|h\|_\mathcal{Y}$ for all $h \in \mathcal{Y}$ for some $C_\mathcal{Y} > 0$. Let
  $L_\mathcal{Y} \colon D(L_\mathcal{Y}) \to \mathcal{Y}$ be an unbounded operator which extends to an unbounded operator $L_\mathcal{Z} \colon D(L_\mathcal{Z}) \to \mathcal{Z}$ be unbounded, i.e.\@ $D(L_\mathcal{Y}) \subseteq D(L_\mathcal{Z})$ and
  $L_\mathcal{Z}\vert_{D(L_\mathcal{Y})} = L_\mathcal{Y}$. Given
  that:
  \begin{enumerate}
  \item $L_\mathcal{Z}$ is the generator of a strongly continuous semigroup $\ee^{L_{\mathcal{Z}}t}$ on $\mathcal{Z}$.
  \item which satisfies
    \begin{equation}
      \label{eq:s1}
      \| \ee^{L_{\mathcal{Z}}t} h\|_\mathcal{Z} \leq C_1 \ee^{-\lambda_1 t}  \|h \|_\mathcal{Z},
      \qquad h\in \mathcal{Z},\:\:t \geq 0
    \end{equation}
    with $C_1 > 0$ and $\lambda_1 \in \R$.
  \item $L_\mathcal{Z} = {A}  + {B}$ with linear operators ${A}$, ${B}$ on
    $\mathcal{Z}$ which satisfy
    \begin{enumerate}
    \item ${A}:\mathcal{Z} \to \mathcal{Y}$ is continuous, i.e.\@ $\| {A} h \|_\mathcal{Y} \leq C_{A} \|h\|_\mathcal{Z}$ for all $h \in \mathcal{Z}$ and $C_{A} > 0$,
    \item ${B}$ is the generator of a strongly continuous semigroup $\ee^{B t}$ on $\mathcal{Y}$ satisfying
      \begin{equation*}
        \norm{\ee^{B t} h}_\mathcal{Y} \leq C_2 \ee^{-\lambda_2 t} \norm{h}_\mathcal{Y},
        \qquad h\in \mathcal{Y}, \:\:t \geq 0
      \end{equation*}
      with $C_2 > 0$, and $\lambda_2 \neq \lambda_1$.
    \end{enumerate}
  \end{enumerate}
  the operator $L_{\mathcal{Z}}$ (and thus $L_{\mathcal{Y}}$) has also a spectral gap on $\mathcal{Y}$, i.e.\@ it satisfies
  \begin{equation}
    \label{eq:s3}
    \norm{\ee^{L_{\mathcal{Z}}t} h}_\mathcal{Y} \leq C \ee^{-\min\{\lambda_1,\lambda_{2}\} t} \norm{h}_\mathcal{Y},
    \qquad h \in \mathcal{Y}, \:\: t \geq 0
  \end{equation}
  for $C = C_2 \left(1 + \frac{C_A C_1 C_{\mathcal{Y}}}{\abs{\lambda_2 - \lambda_1}} \right)$.
\end{theorem}

\begin{proof}
  We can use Duhamel's formula to write
  \begin{equation*}
    \label{eq:Duhamel}
    \ee^{L_{\mathcal{Z}}t} h = \ee^{B t}\,h + \int_0^t
    \ee^{B(t-s)} \left({A}\, \ee^{L_{\mathcal{Z}}s} h\right) \ds \qquad
    \forall h \in \mathcal{Y},\:t\geq 0.
  \end{equation*}
  Thus, for fixed $h \in \mathcal{Y}$ and $t \geq 0$, we  have
  \begin{multline}\label{eq:Vth}
    \|\ee^{L_{\mathcal{Z}}t} h\|_\mathcal{Y}  \leq \|\ee^{B t} h\|_{\mathcal{Y}}
    + \int_0^t \norm{\ee^{B(t-s)} \bigl({A}\,\ee^{L_{\mathcal{Z}}s} h\bigr)}_\mathcal{Y} \ds
    \\
    \leq C_2 e^{-\lambda_2 t} \|h\|_\mathcal{Y}
    +  C_2 \int_0^t e^{-\lambda_2 (t-s)}
    \| {A}\,\ee^{L_{\mathcal{Z}}s} h \|_\mathcal{Y} \ds
    \\
    \leq
    C_2 e^{-\lambda_2 t} \|h\|_\mathcal{Y}
    +  C_2 C_A \int_0^t e^{-\lambda_2 (t-s)}
    \| \ee^{L_{\mathcal{Z}}s} h \|_{\mathcal{Z}} \ds
    \\
    \leq
    C_2 e^{-\lambda_2 t} \|h\|_\mathcal{Y}
    +  C_2 C_A C_1 \| h \|_{\mathcal{Z}}
    \int_0^t e^{-\lambda_2 (t-s)} e^{-\lambda_1 s}
    \ds
    \\
    =
    C_2 e^{-\lambda_2 t} \|h\|_\mathcal{Y}
    +  C_2 C_A C_1 \| h \|_{\mathcal{Z}}
    \frac{1}{\lambda_2 - \lambda_1}(e^{- \lambda_1 t} -  e^{-\lambda_2 t})
    \\
    \leq
    C_2 \left(1 + \frac{C_A C_1 C_{\mathcal{Y}}}{\abs{\lambda_2 - \lambda_1}} \right)
    e^{-\min\{\lambda_1,\lambda_{2}\} t} \|h\|_\mathcal{Y}.
  \end{multline}
  This shows the result.
\end{proof}

\subsection{Extension of the spectral gap}

For convenience we recall \cite[Theorem~3.1]{Canizo2013Exponential}
which allows to extend the spectral gap from one Banach space to a
larger one.

\begin{theorem}\label{Thm:sg:extension}
  Consider two Banach spaces $\mathcal{Y}\subset\mathcal{Z}$ with
  corresponding norms $\norm{\cdot}_{\mathcal{Y}}$ and
  $\norm{\cdot}_{\mathcal{Z}}$ and such that
  $\norm{h}_{\mathcal{Z}}\leq C_{\mathcal{Y}}\norm{h}_{\mathcal{Y}}$
  for all $h\in \mathcal{Y}$. Let
  $L_{\mathcal{Y}}\colon D(L_{\mathcal{Y}})\to \mathcal{Y}$ be an
  unbounded operator which extends to an unbounded operator
  $L_{\mathcal{Z}}\colon D(L_{\mathcal{Z}})\to \mathcal{Z}$, i.e.\@
  $D(L_{\mathcal{Y}})\subset D(L_{\mathcal{Z}})$ and
  $L_{\mathcal{Z}}|_{D(L_{\mathcal{Y}})}=L_{\mathcal{Y}}$. Given that
  \begin{enumerate}
  \item $L_{\mathcal{Y}}$ is the generator of a strongly continuous semigroup $\ee^{L_{\mathcal{Y}}t}$ on $\mathcal{Y}$
  \item which satisfies
    \begin{equation*}
      \norm{\ee^{L_{\mathcal{Y}}t}h}_{\mathcal{Y}}\leq C_{1}\ee^{-\lambda_{1}t}\norm{h}_{\mathcal{Y}} \qquad \text{for }h\in\mathcal{Y} \text{ and } t\geq0
    \end{equation*}
    with $C_{1}>0$ and $\lambda_{1}\in\R$
  \item $L_{\mathcal{Z}}=\mathcal{A}+\mathcal{B}$ with linear operators $\mathcal{A},\mathcal{B}$ on $\mathcal{Z}$ which satisfy
    \begin{enumerate}
    \item $\mathcal{A}\colon \mathcal{Z}\to\mathcal{Y}$ is continuous, i.e.\@ $\norm{\mathcal{A} h}_{\mathcal{Y}}\leq C_{\mathcal{A}}\norm{h}_{\mathcal{Z}}$ for all $h\in\mathcal{Z}$ and $C_{\mathcal{A}}>0$,
    \item $\mathcal{B}$ is the generator of a strongly continuous  semigroup $\ee^{\mathcal{B}t}$ on $\mathcal{Z}$ satisfying
      \begin{equation*}
        \norm{\ee^{\mathcal{B}t} h}_{\mathcal{Z}}\leq C_{2}\ee^{-\lambda_{2} t}\norm{h}_{\mathcal{Z}} \qquad \text{for all }h\in \mathcal{Z}\text{ and }t\geq 0
      \end{equation*}
      with $C_{2}>0$ and $\lambda_{2}>\lambda_{1}$
    \end{enumerate}
  \end{enumerate}
  the operator $L_{\mathcal{Z}}$ is the generator of a strongly continuous  semigroup $\ee^{L_{\mathcal{Z}}t}$ on $\mathcal{Z}$ which extends $\ee^{L_{\mathcal{Y}}t}$ and satisfies
  \begin{equation*}
    \norm{\ee^{L_{\mathcal{Z}}t} h}_{\mathcal{Z}}\leq C\ee^{-\lambda_{1} t}\norm{h}_{\mathcal{Z}} \qquad \text{for all }h\in\mathcal{Z}\text{ and }t\geq0
  \end{equation*}
  with $C=C_{2}+C_{\mathcal{Y}}C_{1}C_{2}C_{\mathcal{A}}(\lambda_{2}-\lambda_{1})^{-1}$.
\end{theorem}

\section{Spectral gap for the constant kernel in weighted \texorpdfstring{$L^1$}{L1} spaces}
\label{sec:gap_L1}

Our overall plan for the linearised coagulation operator consists in
showing that we can restrict the known $H^{-1}(e^{\mu x})$ spectral
gap to the smaller Hilbert space $L^{2}(\ee^{\mu x})$ and from there
extend to the space $L^1((1+x)^{k})$ with sufficiently large $k>0$,
using the techniques from the previous section. In this section, we
will transfer the spectral gap for $\L_{0}$ which has been obtained in
\cite{CMM10} for the class of spaces $H^{-1}(\ee^{\mu x})$ to the
spaces of the form $L^{1}((1+x)^{k})$. For this, we will rely on the
restriction/extension methods recapitulated in
Section~\ref{Sec:tools:sg}.

In~\cite{CMM10} a spectral gap for $\L_{0}$ was obtained in
$H^{-1}(\ee^{\mu x})$. Precisely, we recall from \cite[Prop.\@ 3.11 \&
Lem.\@ 3.12]{CMM10} the following result:

\begin{theorem}\label{Thm:sg:prim:L2}
  For any $\mu\in (0,1)$, the operator $\L_{0}$ as given
  by~\eqref{eq:linearised3} has a spectral gap of size $1$ in
  $H^{-1}(\ee^{\mu x})$, that is: on this space it generates a
  strongly continuous semigroup $\ee^{\L_{0} t}$ satisfying
  \begin{equation*}
    \norm{\ee^{\L_{0} t}h}_{H^{-1}(\ee^{\mu x})}
    \leq
    \norm{h}_{H^{-1}(\ee^{\mu x})}\ee^{-t} \quad \text{for }t\geq 0
  \end{equation*}
  and all $h\in H^{-1}(\ee^{\mu x})\cap
  \{\int_{0}^{\infty}xh(x)\dx=0\}$.
\end{theorem}

\begin{remark}
  In~\cite{CMM10} the constant kernel was chosen to be $K\equiv 1$,
  but the above result is adapted to our choice $K\equiv 2$ which
  seems to be more common.
\end{remark}

\subsection{Restriction of the spectral gap to weigthed \texorpdfstring{$L^2$}{L2} spaces}

In this subsection, we will prove the following proposition which
states that the spectral gap for $\L_{0}$ in $H^{-1}(\ee^{\mu x})$ can
be restricted to the subspace $L^{2}(\ee^{\mu x})$.

\begin{proposition}\label{Prop:restriction:spectral:gap}
  The operator $\L_{0}$ as given by~\eqref{eq:linearised3} generates a strongly continuous semigroup $\ee^{\L_{0}t}$ on $L^{2}(\ee^{\mu x})$ and for each $\lambda\in(-\infty,1/2]$ there exists $C_{\lambda}>0$ such that
  \begin{equation*}
    \norm{\ee^{\L_{0}t}h_{0}}_{L^{2}(\ee^{\mu x})}\leq C_{\lambda}\ee^{-\lambda t}\norm{h_{0}}_{L^{2}(\ee^{\mu x})}\qquad \text{for all } h_{0}\in L^{2}(\ee^{\mu x}) \text{ with } \int_{0}^{\infty}xh_{0}(x)\dx=0
  \end{equation*}
  and all $t\geq 0$.
\end{proposition}

\begin{proof}
  The proof follows from an application of Theorem~\ref{thm:restriction:NEW}. For this, we choose $\mathcal{Z}=H^{-1}(\ee^{\mu x})\cap\{\int_{0}^{\infty}xf(x)\dx=0\}$ and $\mathcal{Y}=L^{2}(\ee^{\mu x})\cap\{\int_{0}^{\infty}xf(x)\dx=0\}$. Moreover, $\L_{\mathcal{Y}}$ and $\L_{\mathcal{Z}}$ are both given as unbounded operators by the expression~\eqref{eq:linearised3} on the respective spaces (see also Remark~\ref{Rem:operator:extension} below).
  
  \begin{enumerate}
  \item According to \cite[Proposition~3.11]{CMM10} (see Theorem~\ref{Thm:sg:prim:L2} above) the operator $\L_{\mathcal{Z}}$ generates a strongly continuous semigroup $\ee^{\L_{\mathcal{Z}}t}$ on the space $H^{-1}(\ee^{\mu x})\cap\{\int_{0}^{\infty}xf(x)\dx=0\}$.
  \item From \cite[Lemma~3.12]{CMM10} (see also Theorem~\ref{Thm:sg:prim:L2} above) we have that
    \begin{equation*}
      \norm{\ee^{\L_{\mathcal{Z}}t}h}_{H^{-1}(\ee^{\mu x})}\leq C_{1}\norm{h}_{H^{-1}(\ee^{\mu x})}\ee^{-t}\qquad \text{for }h\in H^{-1}(\ee^{\mu x})\cap\left\{\int_{0}^{\infty}xh(x)\dx=0\right\}
    \end{equation*}
    and all $t\geq 0$.
  \item We split the operator $\L_{\mathcal{Z}}$ as follows: $\L_{\mathcal{Z}}=A+B$ with
    \begin{equation*}
      \begin{aligned}
        A[h](x)&=-2H(x)+2\int_{0}^{x}H(y)\ee^{-(x-y)}\dy=A_1[h](x)+A_2[h](x)\\
        B[h](x)&=x\del_{x}h(x).
      \end{aligned}
    \end{equation*}
    Note that we use here the notation $H(x)=\int_{x}^{\infty}h(y)\dy$ for the primitive of $h$.
    \begin{enumerate}
    \item \label{It:A:bounded} We have to show that $A\colon \mathcal{Z}\to\mathcal{Y}$ is bounded. Since both $\L_{\mathcal{Z}}$ and $B$ preserve the constraint, the same is true for $A$ and thus it suffices to show that $A\colon H^{-1}(\ee^{\mu x})\to L^{2}(\ee^{\mu  x})$ is bounded. Obviously, we have $\norm{A_{1}[h]}_{L^{2}(\ee^{\mu x})}\leq 2 \norm{h}_{H^{-1}(\ee^{\mu x})}$. Thus, it suffices to consider $A_{2}$:
      \begin{multline*}
        \norm{A_{2}[h]}_{L^{2}(\ee^{\mu x})}^{2}=4\int_{0}^{\infty}\biggl(\int_{0}^{x}H(y)\ee^{-(x-y)}\dy\biggr)^{2}\ee^{\mu x}\dx\\*
        =4\int_{0}^{\infty}\biggl(\int_{0}^{x}H(y)\ee^{\frac{\mu}{2}y}\ee^{\left(\frac{\mu}{2}-1\right)(x-y)}\dy\biggr)^{2}\dx=4\norm*{\bigl(H(\cdot)\ee^{\frac{\mu}{2}\cdot}\bigr)\ast \ee^{\left(\frac{\mu}{2}-1\right)\cdot}}^{2}_{L^{2}(\dx)}.
      \end{multline*}
      Here, $\ast$ denotes the convolution given by $(f\ast g)(x)=\int_{0}^{x}f(x-y)g(y)\dy$ for $x\in(0,\infty)$. From Young's inequality for convolutions, we thus deduces
      \begin{equation*}
        \norm{A_{2}[h]}_{L^{2}(\ee^{\mu x})}\leq 4\norm{h}_{H^{-1}(\ee^{\mu x})}\norm{\ee^{\left(\frac{\mu}{2}-1\right)\cdot}}_{L^{1}(0,\infty)}=\frac{8}{2-\mu}\norm{h}_{H^{-1}(\ee^{\mu x})}.
      \end{equation*}
    \item According to Lemma~\ref{Lem:diff:sg:NEW} the operator $B$ generates a strongly continuous semigroup $\ee^{B t}$ on $L^{2}(\ee^{\mu x})\cap\{\int_{0}^{\infty}xf(x)\dx=0\}$ which satisfies
      \begin{equation*}
        \norm{\ee^{B t} h}_{L^{2}(\ee^{\mu x})}\leq \norm{h}_{L^{2}(\ee^{\mu x})}\ee^{-\frac{1}{2}t} \qquad \text{for all }h\in L^2(\ee^{\mu x})\cap\left\{\int_{0}^{\infty}xf(x)\dx=0\right\}
      \end{equation*}
      and $t\geq 0$.
    \end{enumerate}
  \end{enumerate}
  Thus, according to Theorem~\ref{thm:restriction:NEW} the claim follows.
\end{proof}

\begin{remark}\label{Rem:operator:extension}
  To be able to use Theorem~\ref{thm:restriction:NEW} in the previous
  proof, we have to make sure that $\L_{\mathcal{Z}}$ is an extension
  of $\L_{\mathcal{Y}}$. This follows from the following
  consideration: The proof of
  Proposition~\ref{Prop:restriction:spectral:gap} (i.e.\@
  \ref{It:A:bounded}) shows that $A$ is bounded from $\mathcal{Z}$
  into $\mathcal{Y}$. According to Lemma~\ref{Lem:cont:emb:L2:spaces}
  it is thus in particular bounded from $\mathcal{Z}$ into itself as
  well as from $\mathcal{Y}$ to
  itself. \Cref{Lem:ext:core:general,Lem:diff:sg:NEW} thus ensure that
  both $\L_{\mathcal{Z}}$ and $\L_{\mathcal{Y}}$ are generators of
  strongly continuous semigroups with common core
  $C_{c}^{\infty}(0,\infty)$ on which the two operators coincide.
\end{remark}

\subsection{Extension of the spectral gap to weighted \texorpdfstring{$L^1$}{L1} spaces}

From Proposition~\ref{Prop:restriction:spectral:gap} we know that
$\L_{0}$ has a spectral gap in $L^{2}(\ee^{\mu x})$ subject to the
constraint that $\int_{0}^{\infty}xf(x)\dx=0$. In this subsection we
will now prove, using Theorem~\ref{Thm:sg:extension}, that the latter
can be extended to $L^1((1+x)^{k})$. This is the main spectral gap
result that will be used in the rest of this paper:

\begin{theorem}
  \label{thm:gap-L0}
  Take $k > 2$. The semigroup $e^{t \L_{0}}$ defined on the space
  $L^{1}((1+x)^{k})$ (see \ref{thm:semigroups-defined}) has a spectral
  gap, in the sense that there is $C = C(k) > 0$ such that
  \begin{equation*}
    \norm{\ee^{\L_{0}t}h_{0}}_{k}\leq
    C \ee^{-\frac12 t}\norm{h_{0}}_{k}\qquad \text{for all $t
      \geq 0$,}
  \end{equation*}
  for all $h_{0}\in L^{1}((1+x)^{k})$ with
  $\int_{0}^{\infty}xh_{0}(x)\dx=0$. In particular, by the Hille-Yosida
  theorem we see that
  \begin{equation*}
    \|\L_0(h)\|_{L^1_k} \geq \frac{1}{2 C} \|h\|_{L^1_k} :=  \frac{1}{M} \|h\|_{L^1_k}
  \end{equation*}
  for all $h_{0}$ in the domain of $\L_0$ with
  $\int_{0}^{\infty}xh_{0}(x)\dx=0$
\end{theorem}

The proof of this statement will rely on an application of
Theorem~\ref{Thm:sg:extension}. Precisely, we choose
$\mathcal{Y}=L^{2}(\ee^{\mu x})$ and $\mathcal{Z}=L^{1}((1+x)^{k})$
and we will verify the following steps:
\begin{enumerate}
\item \label{It:ext:1:NEW} $\L_{0}$ given by~\eqref{eq:linearised2} generates a strongly continuous semigroup $\ee^{\L_{0}t}$ on $\mathcal{Y}$.
\item \label{It:ext:2:NEW} For each $\lambda_{1}\leq 1/2$, this semigroup satisfies $\norm{\ee^{\L_{0} t}h}_{L^{2}(\ee^{\mu x})}\leq C_{1}\ee^{-\lambda_{1} t}\norm{h}_{L^{2}(\ee^{\mu x})}$ for all $h\in L^{2}(\ee^{\mu x})\cap\{\int_{0}^{\infty}xh(x)\dx=0\}$ and $t\geq 0$.
\item \label{It:ext:2:5:NEW} There exists a splitting $\L_{0}=\A+\B$ such that:
  \begin{enumerate}
  \item \label{It:ext:3:NEW} $\A\colon L^{1}_{k}\to L^{2}(\ee^{\mu x})$ is continuous, i.e.\@ $\norm{\A h}_{L^{2}(\ee^{\mu x})}\lesssim \norm{h}_{L^{1}_{k}}$ for all $h\in L^{1}_{k}$ and $\int_{0}^{\infty}x\A[h](x)\dx=0$ for all $h\in L_{k}^{1}$ with $\int_{0}^{\infty}xh(x)\dx=0$.
  \item \label{It:ext:4:NEW} $\B$ generates a strongly continuous -semigroup $\ee^{\B t}$ on $L^{1}_{k}$ satisfying
    \begin{equation*}
      \norm{\ee^{\B t}h_{0}}_{L^{1}_{k}}\leq C_{3}\ee^{-t}\norm{h_{0}}_{L^{1}_{k}} \text{ for all } h_{0}\in L_{k}^{1} \text{ with }\int_{0}^{\infty}xh_{0}(x)\dx=0.
    \end{equation*}
  \end{enumerate}
\end{enumerate}

In order to simplify the structure of the actual proof, we collect
first several preparatory results while the proof of
Theorem~\ref{thm:gap-L0} will then be given at the end of this
section.

We will choose the following splitting $\L_{0}=\A+\B$  of the operator $\L_{0}$ with
\begin{equation}\label{eq:sg:ext:split:A:B}
  \begin{aligned}
    \A[h](x)&=2\ee^{-x}\int_{0}^{x}h(y)\chi_{\{y\leq R\}}(\ee^{y}-1)\dy-2\ee^{-x}\int_{x}^{\infty}h(y)\dy\\
    &\qquad+\int_{0}^{\infty}2z\ee^{-z}\int_{0}^{z}h(y)\chi_{\{y>R\}}(\ee^{y}-1)\dy\dz\ee^{-x}=\A_1+\A_2+\A_3\\
    \B[h](x)&=xh'(x)+ 2\ee^{-x}\int_{0}^{x}h(y)\chi_{\{y>R\}}(\ee^{y}-1)\dy\\*
    &\qquad-\int_{0}^{\infty}2z\ee^{-z}\int_{0}^{z}h(y)\chi_{\{y>R\}}(\ee^{y}-1)\dy\dz\ee^{-x}=\B_1+\B_2+\B_3
  \end{aligned}
\end{equation}
where $R$ is a sufficiently large constant which has to be fixed in the proof of Theorem~\ref{thm:gap-L0} below. The last expression on the right-hand side of $\A$ and $\B$ ensures that $\int_{0}^{\infty}x\A[h](x)\dx=0=\int_{0}^{\infty}x\B[h](x)\dx$. At this point, we also exploit that $\L_{0}$ preserves this constraint, i.e.\@ we construct $\B$ such that the first moment is zero which implies that the same is automatically true for $\A$ since $\L_{0}=\A+\B$.

The following three lemmata provide estimates on auxiliary integrals which will turn out to be useful during subsequent computations.

\begin{lemma}\label{Lem:asymptotics:1}
  Let $k\in \R$. For each $\beta>1$ there exists $R_{\beta}>0$ such that 
  \begin{equation*}
    \int_{y}^{\infty}\ee^{-x}(1+x)^{k}\dx\leq \beta \ee^{-y}(1+y)^{k}
  \end{equation*}
  if $y\geq R_{\beta}$.
\end{lemma}

\begin{proof}
  An application of l'H{\^o}pital's rule yields $\int_{y}^{\infty}\ee^{-x}(1+x)^{k}\dx\sim \ee^{-y}(1+y)^{k}$ as $y\to\infty$ from which the claim directly follows.
\end{proof}

\begin{lemma}\label{Lem:gamma:1}
  For each $k\geq 0$ we have the estimate
  \begin{equation*}
    \int_{0}^{\infty}(x+1)^{k}\ee^{-x}\dx\leq \frac{\Gamma(k+1)}{\ee}.
  \end{equation*}
\end{lemma}

\begin{proof}
  The definition of $\Gamma(\cdot)$ together with the change of variables $x\mapsto x-1$ yields
  \begin{equation*}
    \int_{0}^{\infty}(x+1)^{k}\ee^{-x}\dx=\int_{1}^{\infty}x^{k}\ee^{-x-1}\dx=\frac{1}{\ee}\int_{1}^{\infty}x^{(k+1)-1}\ee^{-x}\dx\leq \frac{\Gamma(k+1)}{\ee}.
  \end{equation*}
\end{proof}

\begin{lemma}\label{Lem:asymptotics:2}
  For each $k\geq 0$ there exists a constant $C_{k}>0$ such that
  \begin{equation*}
    \int_{y}^{\infty}\ee^{-x}(1+x)^{k}\dx\leq C_{k} \ee^{-y}(1+y)^{k}
  \end{equation*}
  for all $y>0$.
\end{lemma}

\begin{proof}
  Due to Lemma~\ref{Lem:asymptotics:1} there exists $R_{2}>0$ such that $\int_{y}^{\infty}\ee^{-x}(1+x)^{k}\dx\leq 2 \ee^{-y}(1+y)^{k}$ if $y\geq R_{2}$. Moreover, if $y\leq R_{2}$ we deduce together with Lemma~\ref{Lem:gamma:1} that
  \begin{equation*}
    \int_{y}^{\infty}\ee^{-x}(1+x)^{k}\dx\leq \frac{\Gamma(k+1)}{\ee}\leq \frac{\Gamma(k+1)}{\ee}\ee^{R_{2}}(1+y)^{k}\ee^{-y}.
  \end{equation*}
  Combining both estimates, the claim follows with $C_{k}=\max\{2,\Gamma(k+1)\ee^{R_{2}}/\ee\}$.
\end{proof}

The next lemma shows that the operators $\B_2$ and $\B_3$ are bounded.

\begin{lemma}\label{Lem:ext:B2:B3:bounded}
  For any $\beta>1$ there exists $R_{\beta}>0$ such that the operators
  $\B_{2},\B_{3}\colon L^{1}((1+x)^{k})\to L^{1}((1+x)^{k})$ as
  defined in~\eqref{eq:sg:ext:split:A:B} are bounded with
  \begin{equation*}
    \norm{\B_{2}h}_{L^1_k}\leq 2\beta \norm{h}_{L^1_k}\qquad \text{and}\qquad \norm{\B_{3}h}_{L^1_k}\leq \frac{2\Gamma(k+1)}{\ee (R+1)^{k-1}}\norm{h}_{L^1_k}
  \end{equation*}
  if $R>R_{\beta}$. Moreover, there exists $C_{k}>0$ such that $\norm{\B_{2}h}_{L_{k}^{1}}\leq C_{k}\norm{h}_{L_{k}^{1}}$ for all $R\geq 0$.
\end{lemma}

\begin{proof}
  We first consider the correction term $\A_3=-\B_3$ given in~\eqref{eq:sg:ext:split:A:B} which we rewrite by means of Fubini's theorem and the relation $\int_{y}^{\infty}z\ee^{-z}\dz=(y+1)\ee^{-y}$ which yields
  \begin{multline*}
    \int_{0}^{\infty}2z\ee^{-z}\int_{0}^{z}h(y)\chi_{\{y>R\}}(\ee^{y}-1)\dy\dz=2\int_{0}^{\infty}h(y)\chi_{\{y>R\}}(\ee^{y}-1)\int_{y}^{\infty}z\ee^{-z}\dz\dy\\*
    =2\int_{R}^{\infty}h(y)(y+1)(1-\ee^{-y})\dy=2\int_{R}^{\infty}h(y)(y+1)(1-\ee^{-y})\dy.
  \end{multline*}
  From this, we deduce in particular the estimate 
  \begin{multline}\label{eq:est:proj:corr}
    \abs*{\int_{0}^{\infty}2z\ee^{-z}\int_{0}^{z}h(y)\chi_{\{y>R\}}(\ee^{y}-1)\dy\dz}\leq 2\int_{R}^{\infty}\abs{h(y)}(y+1)^{k}(y+1)^{1-k}(1-\ee^{-y})\dy\\*
    \leq 2(R+1)^{1-k}\int_{R}^{\infty}\abs{h(y)}(y+1)^{k}\dy\leq 2(R+1)^{1-k}\norm*{h}_{L^1_k}.
  \end{multline}
  Thus, together with \cref{eq:sg:ext:split:A:B,Lem:gamma:1} we immediately get
  \begin{equation*}
    \norm{\B_{3}h}_{L^1_k}\leq \frac{2\Gamma(k+1)}{\ee (R+1)^{k-1}}\norm{h}_{L^1_k}.
  \end{equation*}
  To bound $\B_2$ we note that by means of Fubini's theorem we have 
  \begin{multline*}
    \norm{\B_{2}h}_{L^1_k}\leq 2\int_{0}^{\infty}\ee^{-x}\int_{0}^{x}\abs{h(y)}\chi_{\{y>R\}}(\ee^{y}-1)\dy(1+x)^{k}\dx\\*
    =2\int_{R}^{\infty}\abs{h(y)}(\ee^{y}-1)\int_{y}^{\infty}\ee^{-x}(1+x)^{k}\dx\dy.
  \end{multline*}
  Thus, on the one hand fixing $\beta>1$ we obtain together with Lemma~\ref{Lem:asymptotics:1} that
  \begin{equation*}
    \norm{\B_{2}h}_{L^1_k}\leq 2\beta\int_{R}^{\infty}\abs{h(y)}(1+y)^{k}(1-\ee^{-y})\dy\leq 2\beta\norm{h}_{L^1_k}.
  \end{equation*}
  if $R>R_{\beta}$. On the other hand, Lemma~\ref{Lem:asymptotics:2} yields
  \begin{equation*}
    \norm{\B_{2}h}_{L^1_k}\leq C_{k}\int_{0}^{\infty}\abs{h(y)}(1+y)^{k}(1-\ee^{-y})\dy\leq C_{k}\norm{h}_{L^1_k}.
  \end{equation*}
\end{proof}

We next prove that the operator $\B$ generates a strongly continuous semigroup (which is the first part of \cref{It:ext:4:NEW} above).

\begin{lemma}\label{Lem:ext:B:gen:sg:and:core}
  Under the assumptions of Theorem~\ref{thm:gap-L0} the operator $\B$ as defined in~\eqref{eq:sg:ext:split:A:B} generates a strongly continuous semigroup both on $L_k^1$ and $L_{k}^{1}\cap\{\int_{0}^{\infty}xf(x)\dx=0\}$. Moreover, the space $C_{c}^{\infty}(0,\infty)$ or $C_{c}^{\infty}(0,\infty)\cap\{\int_{0}^{\infty}xh(x)\dx=0\}$ respectively is a core.
\end{lemma}

\begin{proof}
  The statement is a consequence of the Bounded Perturbation Theorem (e.g.\@ \cite[Ch.\@ III, Sec.\@ 1.3]{EnN00}) since $\B_{1}=x\del_{x}$ generates a strongly continuous semigroup according to Lemma~\ref{Lem:diff:sg:NEW} while $\B_2$ and $\B_3$ are bounded linear operators as shown in Lemma~\ref{Lem:ext:B2:B3:bounded}. Moreover, the operator $\B$ has been constructed explicitly to preserve the constraint $\int_{0}^{\infty}xf(x)\dx=0$. The statement on the core is a direct consequence of Lemma~\ref{Lem:ext:core:general}.
\end{proof}

With the preparations above, we can now give the proof of Theorem~\ref{thm:gap-L0}.

\begin{proof}[Proof of Theorem~\ref{thm:gap-L0}]
  As already indicated above, the proof relies on an application of Theorem~\ref{Thm:sg:extension}, i.e.\@ we have to verify \cref{It:ext:1:NEW,It:ext:2:NEW,It:ext:2:5:NEW}.
  We recall from Proposition~\ref{Prop:restriction:spectral:gap} that $\L_{0}$ as given by~\eqref{eq:linearised3} generates a
  $C_0$-semigroup $\ee^{\L_{0}t}$ on $L^{2}(\ee^{\mu x})$
  which satisfies for each $\lambda_{1}\leq 1/2$ that 
  \begin{equation}\label{eq:pf:sg:ext:1}
    \norm{\ee^{\L_{0}t}h_{0}}_{L^{2}(\ee^{\mu x})}\leq C_{\lambda_1}\ee^{-\lambda_{1} t}\norm{h_{0}}_{L^{2}(\ee^{\mu x})}\qquad \text{for all } h_{0}\in L^{2}(\ee^{\mu x})\cap\left\{\int_{0}^{\infty}xh_{0}(x)\dx=0\right\}
  \end{equation}
  and all $t\geq 0$. Moreover, $C_{c}^{\infty}(0,\infty)$ is a core for $\L_{0}$ given by~\eqref{eq:linearised3} (see Remark~\ref{Rem:operator:extension}) and thus, according to Lemma~\ref{Lem:rewriting:L} the generator $\L_{0}$ is equivalently represented by~\eqref{eq:linearised2}, i.e.\@ it generates the same semigroup satisfying~\eqref{eq:pf:sg:ext:1}.
  
  It thus remains to verify Item~\ref{It:ext:2:5:NEW} above while we consider first~\ref{It:ext:3:NEW}. As already noted, the operator $\A$ has been explicitly constructed such that the first moment is zero. Thus, it only remains to prove the continuity, i.e.\@ that $\A$ is regularising. To see this, we consider $\A_1,\A_2$ and $\A_3$ separately. To begin with $\A_1$, we find
  \begin{multline}\label{eq:est:A1}
    \norm{\A_{1}[h]}_{L^{2}(\ee^{\mu x})}^{2}=4\int_{0}^{\infty}\ee^{-2x}\biggl(\int_{0}^{x}h(y)\chi_{\{y\leq R\}}(\ee^{y}-1)\dy\biggr)^{2}\ee^{\mu x}\dx\\*
    \leq 4\biggl(\int_{0}^{R}\abs{h(y)}(\ee^{y}-1)\dy\biggr)^{2}\int_{0}^{\infty}\ee^{-(2-\mu)x}\dx\leq \frac{4\ee^{2R}}{2-\mu}\biggl(\int_{0}^{R}\abs{h(y)}\dy\biggr)^{2}\leq \frac{4\ee^{2R}}{2-\mu} \norm{h}_{L^1_k}^{2}.
  \end{multline}
  For $\A_{2}$ we get similarly
  \begin{multline}\label{eq:est:A2}
    \norm{\A_{2}[h]}_{L^{2}(\ee^{\mu x})}^{2}=4\int_{0}^{\infty}\ee^{-2x}\biggl(\int_{x}^{\infty}h(y)\dy\biggr)^{2}\ee^{\mu x}\dx\leq 4\biggl(\int_{0}^{\infty}\abs{h(y)}^2\dy\biggr)^{2}\int_{0}^{\infty}\ee^{-(2-\mu)x}\dx\\*
    =\frac{4}{2-\mu}\biggl(\int_{0}^{\infty}\abs{h(y)}^2\dy\biggr)^{2}\leq \frac{4}{2-\mu}\norm{h}_{L^1_k}^{2}.
  \end{multline}
  Finally, recalling~\eqref{eq:est:proj:corr} from the proof of Lemma~\ref{Lem:ext:B2:B3:bounded} we have
  \begin{equation}\label{eq:est:A3}
    \norm{\A_{3}[h]}_{L^{2}(\ee^{\mu x})}^{2}\leq 4(R+1)^{2(1-k)}\norm*{h}_{L^1_k}^2\int_{0}^{\infty}\ee^{-(2-\mu)x}\dx=\frac{4}{(2-\mu)(R+1)^{2(k-1)}}\norm*{h}_{L^1_k}^2.
  \end{equation}
  Summarising \cref{eq:est:A1,eq:est:A2,eq:est:A3} we obtain
  \begin{equation*}
    \norm{\A[h]}_{L^{2}(\ee^{\mu x})}^{2}\leq \frac{4}{2-\mu}\Bigl(\ee^{2R}+1+(R+1)^{2(1-k)}\Bigr)\norm*{h}_{L^1_k}^2
  \end{equation*}
  which shows that $\A$ is continuous from $L^{1}_{k}$ to $L^{2}(\ee^{\mu x})$.

  Finally, we prove Item~\ref{It:ext:4:NEW}, i.e.\@ that $\B$ generates a strongly continuous -semigroup $\ee^{\mathcal{B}t}$ on $L^{1}_{k}$ which satisfies $\norm{\ee^{\mathcal{B}t}h_{0}}_{L^{1}_{k}}\leq C_{3}\ee^{-\lambda_{3} t}\norm{h_{0}}_{L^{1}_{k}}$ for all $h_{0}\in L_{k}^{1}$ with $\int_{0}^{\infty}xh_{0}(x)\dx=0$ and $\lambda_{3}\leq 1$. According to Lemma~\ref{Lem:ext:B:gen:sg:and:core} the operator $\B$ generates a strongly continuous semigroup on $L_k^1$ which preserves the constraint on the first moment. Thus, it only remains to prove the indicated estimate on the semigroup and according to Lemma~\ref{Lem:ext:B:gen:sg:and:core} we can restrict to the core $C_{c}^{\infty}(0,\infty)$. Thus, for $h_{0}\in C_{c}^{\infty}(0,\infty)$ let $h=h(x,t)=\ee^{\B t}h_{0}$ such that $\del_{t}h=\B h$. We thus have the relation
  \begin{equation}\label{eq:est:sg:B}
    \del_{t}\norm{h}_{L^1_k}=\int_{0}^{\infty}\bigl(\del_{t}h(x,t)\bigr)\sgn(h(x,t))(1+x)^{k}\dx=\int_{0}^{\infty}\B[h](x)\sgn(h(x,t))(1+x)^{k}\dx.
  \end{equation}
  We require estimates on the right-hand side of this equation. Again, we treat the expressions $\B_1,\B_2$ and $\B_3$ separately and to simplify the notation, we only write $h(x)$, i.e.\@ neglecting the time-dependence, in the following. With $\abs{h}'=h'\sgn(h)$ integration by parts yields
  \begin{multline}\label{eq:est:B1}
    \int_{0}^{\infty}\B_{1}[h](x)\sgn(h(x))(1+x)^{k}\dx=\int_{0}^{\infty}\abs{h(x)}'x(1+x)^{k}\dx\\*
    =-\int_{0}^{\infty}\abs{h(x)}\bigl((1+x)^{k}+kx(1+x)^{k-1}\bigr)\dx=-\norm{h}_{L^1_k}-k\int_{0}^{\infty}\abs{h(x)}x(1+x)^{k-1}\dx.
  \end{multline}
  Next we consider $\B_2$ for which we obtain together with Fubini's theorem and Lemma~\ref{Lem:asymptotics:1} that
  \begin{multline}\label{eq:est:B2}
    \int_{0}^{\infty}\B_{2}[h](x)\sgn(h(x))(1+x)^{k}\dx=2\int_{0}^{\infty}\ee^{-x}\int_{0}^{x}h(y)\chi_{\{y>R\}}(\ee^{y}-1)\dy\sgn(h(x))(1+x)^{k}\dx\\*
    \leq 2\int_{R}^{\infty}\ee^{-x}\int_{R}^{x}\abs{h(y)}(\ee^{y}-1)\dy(1+x)^{k}\dx=2\int_{R}^{\infty}\abs{h(y)}(\ee^{y}-1)\int_{y}^{\infty}\ee^{-x}(1+x)^{k}\dx\dy\\*
    \leq 2\beta \int_{R_{\beta}}^{\infty}\abs{h(y)}(1-\ee^{-y})(1+y)^{k}\dy.
  \end{multline}
  Finally, recalling~\eqref{eq:est:proj:corr} from the proof of Lemma~\ref{Lem:ext:B2:B3:bounded} we estimate $\B_{3}$ together with Lemma~\ref{Lem:gamma:1} as
  \begin{multline}\label{eq:est:B3}
    \int_{0}^{\infty}\B_{3}[h](x)\sgn(h(x))(1+x)^{k}\dx\\*
    \leq 2(R+1)^{1-k}\int_{R}^{\infty}\abs{h(y)}(1+y)^{k}\dy\int_{0}^{\infty}\ee^{-x}\sgn(h(x))(1+x)^{k}\dx\\*
    \leq 2(R+1)^{1-k}\int_{R}^{\infty}\abs{h(y)}(1+y)^{k}\dy\int_{0}^{\infty}\ee^{-x}(1+x)^{k}\dx\\*
    \leq \frac{2\Gamma(k+1)}{\ee(R+1)^{k-1}}\int_{R}^{\infty}\abs{h(y)}(1+y)^{k}\dy.
  \end{multline}
  Summarising \cref{eq:est:B1,eq:est:B2,eq:est:B3} we obtain
  \begin{multline}\label{eq:est:B}
    \int_{0}^{\infty}\B[h](x)\sgn(h(x))(1+x)^{k}\dx\\*
    \leq -\norm{h}_{L^1_k}-k\int_{0}^{\infty}\abs{h(x)}x(1+x)^{k-1}\dx +2\beta \int_{R_{\beta}}^{\infty}\abs{h(y)}(1-\ee^{-y})(1+y)^{k}\dy\\*
    +\frac{2\Gamma(k+1)}{\ee(R_{\beta}+1)^{k-1}}\int_{R_{\beta}}^{\infty}\abs{h(y)}(1+y)^{k}\dy\\*
    \shoveleft{= -\norm{h}_{L^1_k}-k\int_{0}^{R_{\beta}}\abs{h(x)}x(1+x)^{k-1}\dx}\\*
    +\int_{R_{\beta}}^{\infty}\abs{h(x)}(1+x)^{k-1}\Bigl(2\beta(1+x)(1-\ee^{-x})+\frac{2\Gamma(k+1)(1+x)}{\ee(R_{\beta}+1)^{k-1}}-kx\Bigr)\dx.
  \end{multline}
  We fix $\beta>1$ satisfying $k>2\beta$ (notice this is where the
  restriction on the values of $k$ comes into play) and choose then
  $R_{\beta}$ sufficiently large such that
  \begin{equation*}
    \Bigl(2\beta(1+x)(1-\ee^{-x})+\frac{2\Gamma(k+1)(1+x)}{\ee(R_{\beta}+1)^{k-1}}-kx\Bigr)<0 \qquad \text{for all }x\geq R_{\beta}.
  \end{equation*}
  The latter is possible if $R_{\beta}$ is large enough to satisfy for example
  \begin{equation*}
    R_{\beta}>\Bigl(\frac{2\Gamma(k+1)}{\ee(k-2\beta)}\Bigr)^{\frac{1}{k-1}}-1\qquad \text{and}\qquad \Bigl(2\beta+\frac{\Gamma(k+1)}{\ee(1+R_{\beta})^{k-1}}-k\Bigr)(R_{\beta}+1)<-k.
  \end{equation*}
  For this choice of $R_{\beta}$, we deduce from~\eqref{eq:est:B} that
  \begin{equation*}
    \int_{0}^{\infty}\B[h](x)\sgn(h(x))(1+x)^{k}\dx\leq -\norm{h}_{L^1_k}.
  \end{equation*}
  Recalling~\eqref{eq:est:sg:B}, Grönwall's inequality yields the desired estimate on the semigroup generated by $\B$. 
\end{proof}

\begin{remark}
  \label{Rem:operator:extension:2}
  The fact that the operator $\L_{0}$ defined
  by~\eqref{eq:linearised2} as an unbounded operator on $L_{k}^{1}$ is
  an extension of the same expression defined on $L^{2}(\ee^{\mu x})$
  follows by an argument analogous to that in
  Remark~\ref{Rem:operator:extension}. Precisely, from the proof of
  Proposition~\ref{Prop:restriction:spectral:gap} we know that
  $C_{c}^{\infty}(0,\infty)$ is a core for $\L_{0}$ on
  $L^{2}(\ee^{\mu x})$. Since Lemma~\ref{Lem:ext:B2:B3:bounded} and
  the proof of Theorem~\ref{Thm:sg:extension} provide that $\A_{2}$
  and $\B_{2}$ (for $R=0$) are bounded from $L_{k}^{1}$ to itself we
  deduce from \cref{Lem:diff:sg:NEW,Lem:ext:core:general} that
  $C_{c}^{\infty}(0,\infty)$ is also a core for $\L_{0}$ on
  $L_{k}^{1}$ and on this common core, both operators coincide.
\end{remark}

\section{Uniqueness of profiles}
\label{sec:local}

The spectral gap estimates proved in the previous section allow us to
show that small perturbations of the equation for $K=2$ have
essentially the same behaviour, at least when solutions are not far
from the self-similar profile $e^{-x}$ for $K=2$. We gather all local
results of this type in this section.

\subsection{Local stability of profiles}
\label{sec:local_stability}

We call \emph{stability of the self-similar profiles} with respect to
the perturbation we are considering the property that for small
$\eps$ the unit-mass profiles are close to the unique unit-mass
profile $G_0$ for $\eps = 0$. We prove now a \emph{local} version
of this result, which states that this is true provided the profiles
are contained in a ball of a specific radius around $G_0$. Global
versions are given in Section \ref{sec:stability-profiles}, but the
advantage of the local statements we give now is that they give an
optimal rate of stability, and they use only the properties of the
linearised operator $\L_{0}$.

Our first observation is that the nonlinear operators defining the
equation for $\eps = 0$ and its perturbation are not far from each
other in the $\|\cdot\|_{L^1_k}$ norm:
\begin{lemma}
  \label{lem:nonlinear_operator_close}
  Denote by $\No_\eps$ the self-similar Smoluchowski operator with
  kernel $K_\eps$; that is,
  \begin{equation*}
    \No_\eps(f) := 2f + x \partial_x f + \C_{\eps}(f,f).
  \end{equation*}
  For any $k \geq 0$ and any $f \in L^1_k$ we have
  \begin{equation*}
    \|\No_\eps(f) - \No_0(f) \|_{L^1_k} \leq \frac32 \eps \|f\|_{L^1_k}^2.
  \end{equation*}
\end{lemma}

\begin{proof}
  We have
  \begin{equation*}
    \No_\eps(f) := 2f + x \partial_x f + \C_{\eps}(f,f),
  \end{equation*}
  so
  \begin{equation*}
    \| \No_\eps(f) - N_0(f)\|_{L^1_k}
    = \| \C_{\eps}(f,f) - \C_0(f,f) \|_{L^1_k}
    = \eps \| \C_{W}(f,f) \|_{L^1_k}
    \leq \frac{3}{2} \eps \|f\|_{L^1_k}^2,
  \end{equation*}
  where the last inequality is due to Lemma \ref{Prop:CK:cont}.
\end{proof}

We now give our local result on the stability of profiles:

\begin{lemma}[Local stability of profiles]
  \label{lem:local_stability}
  Take $0 \leq \eps < 1$, let $G_\eps$ be a self similar profile for
  the kernel $K_\eps$, and assume that
  $\|G_\eps - G_0\| \leq \frac{1}{6M}$, where $M > 0$ is the one in
  \cref{thm:gap-L0}. Take $k \geq 0$. There exists an (explicit)
  constant $M_1 = M_1(k) > 0$ such that
  \begin{equation*}
    \|G_\eps - G_0\|_{L^1_k}
    \leq \eps M_1.
  \end{equation*}
  
\end{lemma}

\begin{proof}
  Denote $\No_\eps$ the same operator as in Lemma
  \ref{lem:nonlinear_operator_close}, and let $G_\eps$ be any
  self-similar profile for the kernel $K_\eps$. Since
  $\No_\eps(G_\eps) = 0$ we have
  \begin{equation}
    \label{eq:sp1}
    \|\No_0(G_\eps)\|_{L^1_k}
    = \|\No_0(G_\eps) - \No_\eps(G_\eps)\|_{L^1_k}
    \leq \frac32 \eps \|G_\eps\|_{L^1_k}^2,
  \end{equation}
  where we have used Lemma \ref{lem:nonlinear_operator_close}. Now
  \begin{equation*}
    \No_0(G_\eps) =
    \L_{0}(G_\eps - G_0) + \C_0(G_\eps - G_0, G_\eps - G_0),
  \end{equation*}
  so
  \begin{multline*}
    \|\No_0(G_\eps)\|_{L^1_k}
    =
    \| \L_{0}(G_\eps - G_0) + \C_0(G_\eps - G_0, G_\eps - G_0) \|_{L^1_k}
    \\
    \geq
    \| \L_{0}(G_\eps - G_0) \|_{L^1_k}
    - \|\C_0(G_\eps - G_0, G_\eps - G_0) \|_{L^1_k}
    \\
    \geq
    \frac{1}{M} \| G_\eps - G_0 \|_{L^1_k}
    -  3 \|G_\eps - G_0 \|_{L^1_k}^2.
  \end{multline*}
  Together with \eqref{eq:sp1} this gives
  \begin{equation*}
    \frac{1}{M} \|G_\eps - G_0\|_{L^1_k}
    \leq
    \frac32 \eps \|G_\eps\|_{L^1_k}^2 +
    3 \|G_\eps - G_0 \|_{L^1_k}^2.
  \end{equation*}
  If we assume that $\|G_\eps - G_0\|_{L^1_k} \leq \frac{1}{6M}$
  then this implies
  \begin{equation*}
    \frac{1}{2M} \|G_\eps - G_0\|_{L^1_k}
    \leq
    \frac32 \eps \|G_\eps\|_{L^1_k}^2.
  \end{equation*}
  \cref{Prop:L1:est:profile} then shows that the right hand side is
  finite and depends only on $k$.
\end{proof}

As an immediate consequence of Lemma~\ref{lem:local_stability} and
\cref{Prop:profile_stability}, we then also obtain the following
global stability result (notice that in Section
\ref{sec:stability-profiles} the constant in
\cref{Prop:profile_stability} was explicitly estimated, so the $D_k$
in the following result is constructive). Also, we remark that a
global result like \cref{Prop:profile_stability} is essential here,
since the stability of all possible solutions to the self-similar
equation cannot be proved by studying only its linearisation.

\begin{corollary}[Global stability of profiles]\label{Cor:global:stability}
  For each $k\geq 0$ there exists a constant $D_{k}$ such that each self-similar profile $G_{\eps}$ for the kernel $K_{\eps}$ with $0\leq \eps\leq 1$ satisfies
  \begin{equation*}
    \norm{G_{\eps}-G_{0}}_{L^{1}_{k}}\leq D_{k}\eps.
  \end{equation*}
\end{corollary}

\subsection{Uniqueness}
\label{sec:local_uniqueness}

We first show a \emph{local uniqueness} result which states that
self-similar profiles are unique for small epsilon, provided they are
within a certain distace of $G_0$. As in the previous results in this
section, we prefer to state this local result because it depends only
on linearisation arguments involving the operator $\L_{0}$. Using the
stability results in Section \ref{sec:stability-profiles}, it
immediately gives a full uniqueness result.

In order to state our local uniqueness result we first show that the
perturbed linear operators $\L_\eps$ around a self-similar profile
$G_\eps$ which is in a certain ball around $G_0$, also have a spectral
gap in the $L^1_k$ spaces for small $\eps$. This is not strictly
needed, but it makes the later proof a bit easier. Notice that the
operators $\L_\eps$ are just bounded perturbations of the operator
$\L_0$, so it is understood that they are defined in the same way as
in Theorem \ref{thm:semigroups-defined}, with the same domain.

\begin{lemma}
  \label{lem:L_e-L_0-local}
  Take $k \geq 0$ and $0 < \eps < 1$, and call $\L_\eps$ the
  linearised self-similar Smoluchowski operator in the space $L^1_k$,
  with kernel $K_\eps$, around a given self-similar profile $G_\eps$
  with mass $1$. There is an explicit constant $M_2 = M_2(k)$ such
  that
  \begin{equation*}
    \| \L_\eps - \L_0 \| \leq \eps M_2 
  \end{equation*}
\end{lemma}

\begin{proof}
  In terms of $\C_{K}$ as given in~\eqref{eq:sym:coag:op} the operators
  $\L_{\eps}$ and $\L_{0}$ read
  \begin{equation*}
    \begin{aligned}
      \L_{0}[h](x)&=2h+xh'+2\C_{0}(G_0,h)\\
      \L_{\eps}[h](x)&=2h+xh'+2\C_{\eps}(G_{\eps},h).
    \end{aligned}
  \end{equation*}
  This yields
  \begin{equation*}
    \L_{\eps}[h](x)-\L_{0}[h](x)
    = 2\C_{\eps}(G_{\eps},h)-2\C_{0}(G_0,h)
    = 2\C_{0}(G_{\eps}-G_0,h)+2\eps \C_{W}(G_{\eps},h).
  \end{equation*}
  With \cref{Prop:CK:cont,Cor:global:stability,Prop:L1:est:profile} we thus deduce
  \begin{equation*}
    \norm{\L_{\eps}[h]-\L_{0}[h]}_{L^{1}_k}
    \leq 6 \norm{G_{\eps}-G_0}_{L^1_k} \norm{h}_{L^1_k}
    +3\eps \norm{G_{\eps}}_{L^1_k}\norm{h}_{L^1_k} \leq \left(6 D_{k}\eps +
      3C_{k}\eps \right) \norm{h}_{L^1_k}.
  \end{equation*}
  According to \cref{Prop:L1:est:profile,Cor:global:stability} the coefficient $M_{2}(k)\vcc=(6D_{k}+3C_{k})$ depends only on $k$, which proves the result.
\end{proof}

\begin{lemma}[Spectral gap of $\L_\eps$]
  \label{lem:gap-epsilon-local}
  Let $\L_\eps$ be the linearised self-similar Smoluchowski operator
  with kernel $K_\eps$, around a given self-similar profile $G_\eps$
  with mass $1$ and $k>2$. Then for
  $\eps<\frac{1}{2CM_{2}}=\vcc\eps_{0}$, where $M_2 = M_2(k)$ is from
  \cref{lem:L_e-L_0-local} and $C=C(k)$ is from
  Theorem~\ref{thm:gap-L0}, the operator $\L_\eps$ has a spectral gap
  in $L^1_k$ of size $1/2-CM_{2}\eps$. That is: with the same $C=C(k)$
  from Theorem~\ref{thm:gap-L0} we have
  \begin{equation*}
    \|e^{t \L_\eps} h_0\|_{L^1_k}
    \leq  C\|h_0\|_{L^1_k} e^{-(\frac{1}{2}-CM_{2}\eps) t},
    \qquad t \geq 0
  \end{equation*}
  for all $h_0 \in L^1_k$ with $\int x h_0 = 0$.
\end{lemma}
We sometimes state the result above by saying that, under these
conditions, the linearised operator $\L_\eps$ has a spectral gap in
$L^1_k$ of size $1/2-CM_{2}\eps$.

\begin{remark}
  \label{Rem:gap:epsilon2}
  In particular, under the assumptions of the previous result, by the
  Hille-Yosida theorem we have
  \begin{equation}
    \label{eq:Leps-invertible}
    \| \L_\eps h \|_{L^1_k} \geq \frac{1-2CM_{2}\eps}{2C} \|h\|_{L^1_k}
  \end{equation}
  for all $h \in L^1_k$ with $\int x h(x) \dx = 0$.
\end{remark}

\begin{proof}[Proof of Lemma \ref{lem:gap-epsilon-local}]
  All norms used in this proof are $\|\cdot\|_{L^1_k}$, and we omit
  the subscript to simplify the notation. From Lemma
  \ref{lem:L_e-L_0-local} we have
  \begin{equation}
    \label{eq:L_e-L_0-aux}
    \| \L_\eps - \L_0 \| \leq \eps M_2 =: \delta.
  \end{equation}
  We consider the equation
  \begin{equation*}
    \partial_t h = \L_\eps h = \L_0 h + (\L_\eps - \L_0) h
  \end{equation*}
  and write, using Duhamel's formula and setting
  $h_t := e^{t \L_\eps} h_0$,
  \begin{equation*}
    h_t = e^{t \L_0} h_0
    + \int_0^t e^{(t-s)\L_0}(\L_\eps - \L_0)  h_s \ds.
  \end{equation*}
  Hence, using Theorem \ref{thm:gap-L0} for $\L_0$
  \begin{equation*}
    \|h_t\| \leq C e^{-\frac12 t} \|h_0\|
    + C \int_0^t e^{-\frac12 (t-s)} \|(\L_\eps - \L_0)  h_s \| \ds.
  \end{equation*}
  Hence, from \eqref{eq:L_e-L_0-aux},
  \begin{equation*}
    \|h_t\| \leq C e^{-\frac12  t} \|h_0\|
    + C \delta \int_0^t e^{-\frac12 (t-s)} \| h_s \| \ds.
  \end{equation*}
  Calling $u(t) := \|h_t\| e^{\frac12  t}$ we see that
  \begin{equation*}
    u(t) \leq C \|h_0\| + C \delta \int_0^t u(s) \ds,
  \end{equation*}
  so by Gronwall's Lemma we have $u(t) \leq C \|h_0\| e^{C \delta t}$,
  that is
  \begin{equation*}
    \|h_t\| \leq C \|h_0\| e^{(C \delta - \frac12 ) t}
    ,
    \qquad t \geq 0.
  \end{equation*}
  Recalling~\eqref{eq:L_e-L_0-aux} this shows the claim.
\end{proof}

We can finally give the proof of local uniqueness of the profiles:

\begin{theorem}[Local uniqueness of self-similar profiles]
  \label{thm:uniqueness-local}
  Take any $k > 2$. For all $0 < \eps < \eps_{0}$ (with $\eps_{0}$
  from Lemma \ref{lem:gap-epsilon-local}), Smoluchowski's coagulation
  equation with kernel $K_\eps$ has at most one self-similar profile
  $G_\eps$ with mass $1$ satisfying
  \begin{equation*}
    \|G_\eps - G_0\|_{L^1_k} \leq \frac{1-\eps/\eps_{0}}{12C(1+\eps)}
  \end{equation*}
  with $C=C(k)$ from Theorem~\ref{thm:gap-L0}.
\end{theorem}

\begin{proof}
  Let $\No_\eps$ be the the self-similar Smoluchowski operator with
  kernel $K_\eps$.  Assume we have two different self-similar
  profiles $G_1, G_2$ with mass $1$ for the kernel $K_\eps$:
  \begin{equation*}
    \No_\eps(G_1) = \No_\eps(G_2) = 0,
  \end{equation*}
  and that they both satisfy
  \begin{equation}
    \label{eq:lup1}
    \|G_1 - G_0\|_{L^1_k} \leq \frac{1-\eps/\eps_{0}}{12C(1+\eps)},
    \qquad
    \|G_2 - G_0\|_{L^1_k} \leq \frac{1-\eps/\eps_{0}}{12C(1+\eps)}.
  \end{equation}
  Call $\L_\eps$ the linearised self-similar Smoluchowski operator
  with kernel $K_\eps$, around the profile $G_1$.  Since
  $\No_\eps(f) = \L_\eps(f-G_1) + \C_{2+\eps W}(f-G_1, f- G_1)$,
  \begin{equation*}
    0 = \No_\eps(G_1) - \No_\eps(G_2)
    = \L_\eps(G_2 - G_1) - \C_{2+\eps W}(G_1 - G_2, G_1 - G_2).
  \end{equation*}
  Using \cref{lem:gap-epsilon-local} (see equation
  \eqref{eq:Leps-invertible} in particular) and \cref{Prop:CK:cont},
  \begin{equation*}
    \frac{1-\eps/\eps_{0}}{2C} \| G_1 - G_2 \|
    \leq
    \| \L_\eps(G_2 - G_1) \|
    \leq \frac32 (1+\eps) \|G_1 - G_2\|^2,
  \end{equation*}
  so, since $G_1 \neq G_2$, and assuming always $\eps < 1$,
  \begin{equation*}
    \|G_1 - G_2\| \geq \frac{1-\eps/\eps_{0}}{3C(1+\eps)}
  \end{equation*}
  This contradicts \eqref{eq:lup1}, since
  \begin{equation*}
    \|G_1 - G_2\|
    \leq
    \|G_1 - G_0\| + \|G_0 - G_2\|
    \leq \frac{1-\eps/\eps_{0}}{6C(1+\eps)}.
    \qedhere
  \end{equation*}
\end{proof}

We can then use this local result for any fixed $k > 2$, together with
the stability results in Section \ref{sec:stability-profiles}, to
obtain that there is a unique unit-mass self-similar profile in
$L^1_k$. Since we know from Section \ref{sec:exist-self-simil} that
all profiles must be in $L^1_k$, we immediately obtain a uniqueness
result:
\begin{corollary}[Uniqueness of profiles for small perturbations]
  \label{cor:uniqueness}
  There exists $\eps_1 > 0$ such that for all
  $0 \leq \eps \leq \eps_1$ Smoluchowski's coagulation
  equation with kernel $K_\eps$ has exactly one self-similar
  profile $G_\eps$ with mass $1$.
\end{corollary}

\section{Convergence to equilibrium}
\label{sec:convergence}

\subsection{Local convergence to equilibrium}
\label{sec:local_convergence}

\begin{proposition}[Local exponential convergence to equilibrium]
  \label{prp:local_nonlinear_convergence}
  Take $k > 2$, and consider $\eps_1$ from Corollary
  \ref{cor:uniqueness}. For any $0 \leq \eps < \eps_1$, let $G_\eps$
  be the unique self-similar profile with mass $1$ for the kernel
  $K_\eps$. There exist constants $C_{*}, M, \eps_2, r_1 > 0$
  depending on $k$ only such that for any $0 \leq \eps \leq \eps_2$,
  any solution $f$ to the self-similar Smoluchowski equation
  \eqref{eq:Smol:selfsim} with kernel $K_\eps$ with initial condition
  $f_0 \in L^1_k$ such that
  \begin{equation*}
    \int_0^\infty xf_0(x) \dx = 1,
    \qquad \|f_0 - G_\eps\|_{L^1_k} \leq r_1
  \end{equation*}
  satisfies
  \begin{equation*}
    \| f(t, \cdot) - G_\eps \|_{L^1_k}
    \leq C_{*} e^{-(\frac{1}{2}-M\eps) t} \| f_0 - G_\eps \|_{L^1_k}
    \qquad \text{for all $t \geq 0$.}
  \end{equation*}
\end{proposition}

\begin{proof}
  Since we have information on the spectral properties of the
  linearised operator $\L_\eps$ around the profile $G_\eps$,
  the proof becomes a standard perturbation argument: we write the
  self-similar Smoluchowski equation as
  \begin{equation*}
    \partial_t f = \C_{2+\eps W}(f,f) + 2 f + x \partial_x f
    = \L_\eps(f - G_\eps) + \C_{2+\eps W}(f - G_\eps, f-G_\eps).
  \end{equation*}
  By Duhamel's formula, calling $h := f - G_\eps,$
  \begin{equation*}
    h_t = e^{t \L_\eps} h_0
    + \int_0^t e^{(t-s)\L_\eps} \big( \C_{2+\eps W}(h_s, h_s) \big) \ds,
  \end{equation*}
  so, using \cref{lem:gap-epsilon-local,Prop:CK:cont} and denoting $\lambda_{\eps}\vcc=1/2-CM_{2}\eps$ (with $C$ and $M_2$ from Lemma~\ref{lem:gap-epsilon-local}),
  \begin{multline*}
    \|h_t\|_{L^1_k}
    \leq
    C e^{- \lambda_{\eps} t} \|h_0\|_{L^1_k}
    + C \int_0^t e^{-\lambda_{\eps}(t-s)} \| \C_{2+\eps W}(h_s, h_s) \|_{L^1_k} \ds
    \\
    \leq
    C e^{- \lambda_{\eps} t} \|h_0\|_{L^1_k}
    + 2 C \|K_\eps\|_\infty
    \int_0^t e^{-\lambda_{\eps}(t-s)} \| h_s \|_{L^1_k}^2 \ds.
  \end{multline*}
  If we define $u(t) := \| h(t,\cdot)\|_{L^1_k} e^{\lambda_{\eps} t}$ we have
  \begin{equation*}
    u(t)
    \leq 
    C u(0)
    + 2 C \|K_\eps\|_\infty
    \int_0^t u(s)^2 e^{-\lambda_{\eps} s} \ds.
  \end{equation*}
  Gronwall's lemma applied to this integral inequality then shows that
  \begin{equation*}
    u(t) \leq \left( \frac{1}{C u(0)}
      - \frac{2 C \|K_\eps\|_\infty}{\lambda_{\eps}} (1 - e^{-\lambda_{\eps} t}) \right)^{-1},
  \end{equation*}
  which remains bounded for all $t \geq 0$ if
  \begin{equation*}
    u(0) < \frac{\lambda_{\eps}}{2 C^2 \|K_\eps\|_\infty}.
  \end{equation*}
  For example, if we assume
  \begin{equation*}
    u(0) < \frac{\lambda_{\eps}}{4 C^2 \|K_\eps\|_\infty}
  \end{equation*}
  then
  \begin{equation*}
    u(t) \leq 2 C u(0)
    \qquad \text{for all $t \geq 0$,}
  \end{equation*}
  which implies
  \begin{equation*}
    \| h(t, \cdot) \| \leq 2 C e^{-\lambda_{\eps} t} \|h_0\|,
  \end{equation*}
  which is what we wanted to show.
\end{proof}

\subsection{Convergence to equilibrium in large regions}

If we additionally use our knowledge that solutions to the
unperturbed problem with kernel $K = 2$ converge to equilibrium
globally we can obtain a slight improvement of the above
result. Namely, that the size $R$ of the region in which we have
convergence can be taken as large as one wants, provided $\eps$ is
close enough to zero:

\begin{theorem}[Exponential convergence to equilibrium in large
  regions for small $\eps$]
  Let $k>2$ and $W$ be a bounded kernel of homogeneity $0$, take
  $R > 0$, and take $0 \leq \eps \leq \eps_1$ (with $\eps_1$ the one
  from \cref{cor:uniqueness} ensuring uniqueness of profiles). Denote
  $K_\eps := 2 + \eps W$, and call $G_\eps$ the unique self-similar
  profile with mass $1$ for the kernel $K_\eps$. There exist constants
  $C$ and $M$ (depending only on $k$) and $\eps_3 > 0$ (depending on
  $W$, $R$ and $k$) such that any solution $f$ to the self-similar
  Smoluchowski equation \eqref{eq:Smol:selfsim} with kernel $K_\eps$
  with $0 \leq \eps \leq \eps_3$ and initial condition
  $f_0 \in L_{k}^{1}$ with $f_{0}\geq 0$ almost everywhere and
  \begin{equation*}
    \int_0^\infty xf_0(x) \dx = 1, \qquad \|f_0 - G_\eps\| \leq R
  \end{equation*}
  satisfies
  \begin{equation*}
    \| f(t, \cdot) - G_\eps \| \leq C e^{-(\frac{1}{2}-M\eps) t} \| f_0 -
    G_\eps \|
    \qquad \text{for all $t \geq 0$.}
  \end{equation*}
\end{theorem}

\begin{proof}
  The idea that we want to exploit is that for small $\eps$,
  solutions to our perturbed equation are not too far from solutions
  to the equation for the constant kernel. Since we know that the
  equation for the constant kernel converges to equilibrium
  exponentially fast, we can show that solutions to the perturbed
  equation will eventually fall inside the local region where we can
  apply Proposition \ref{prp:local_nonlinear_convergence}.
  
  For $R > 0$ given, and any $\eps > 0$, take $k > 2$ and any
  nonnegative initial condition $f_0 \in L^1_k$ with mass $1$ and
  $\|f_0 - G_\eps\| \leq R$. Call $f^\eps$ the solution to the
  self-similar Smoluchowski equation with kernel $K_\eps$, and $f$
  the solution to the self-similar Smoluchowski equation with constant
  kernel $K_0 = 2$, both with initial condition $f_0$. From Lemma
  \ref{Lem:L1:evolution:close} we know that these two solutions remain
  close for some time: for some $C_3, C_4 > 0$,
  \begin{equation*}
    \|f^\eps(t,\cdot) - f(t,\cdot) \|_{L^1_k}
    \leq
    C_3 \eps e^{C_4 t}.
  \end{equation*}
  Also, according to Theorem~\ref{Thm:L2:convergence} the solution $f$ converges exponentially fast to $G_0(x) =
  e^{-x}$:
  \begin{equation*}
    \|f(t,\cdot) - G_0\|_{L^1_k} \leq C_5 e^{-\frac{1}{2} t}.
  \end{equation*}
  Hence together with Corollary~\ref{Cor:global:stability}
  \begin{multline*}
    \|f^\eps(t,\cdot) - G_\eps\|_{L^1_k}
    \leq \| f^\eps(t,\cdot) - f(t,\cdot) \|_{L^1_k}
    + \| f(t,\cdot) - G_0\|_{L^1_k}
    + \| G_0 - G_\eps\|_{L^1_k}
    \\
    \leq
    C_3 \eps e^{C_4 t} + C_5 e^{-\frac{1}{2} t} + D_{k}\eps.
  \end{multline*}
  We can choose large enough $t$ (which we call $t_0$), and then small
  enough $\eps$, so that this quantity is less than the $r_1$ in
  \cref{prp:local_nonlinear_convergence}. Then, from
  \cref{prp:local_nonlinear_convergence},
  \begin{equation*}
    \| f^{\eps}(t,\cdot) - G_\eps \|_{L^1_k} \leq
    C e^{-(1/2-M\eps) (t-t_0)} \|f^{\eps}(t_0,\cdot) - G_\eps\|_{L^1_k}
    \qquad
    \text{for all $t \geq t_0$.}
  \end{equation*}
  It is also easy to see that, for some $C_6 > 0$,
  \begin{equation*}
    \|f^{\eps}(t_0,\cdot) - G_\eps\|_{L^1_k}
    \leq
    e^{C_6 t_0} \|f_0 - G_\eps\|_{L^1_k},
  \end{equation*}
  which then gives
  \begin{equation*}
    \| f^{\eps}(t,\cdot) - G_\eps \|_{L^1_k} \leq
    C e^{-(1/2-M\eps) (t-t_0)} e^{C_6 t_0} \|f_0 - G_\eps\|_{L^1_k}
    \qquad
    \text{for all $t \geq t_0$.}
  \end{equation*}
  This shows the result for $t \geq t_0$, and for $t \leq t_0$ we
  can easily obtain
  \begin{equation*}
    \| f^{\eps}(t,\cdot) - G_\eps \|_{L^1_k} \leq
    e^{C_7 t} \|f_0 - G_\eps\|_{L^1_k}
  \end{equation*}
  by similar calculations as in
  \cref{prp:local_nonlinear_convergence}, using that we already know
  from \ref{Prop:self:sim:evolution:existence} that
  $\|f^\eps(t,\cdot)\|_{L^1_k}$ is uniformly bounded for all times. This is
  enough to obtain the result.
\end{proof}

\section*{Acknowledgements}

JAC and ST were supported by project MTM2017-85067-P, funded by the
Spanish government and the European Regional Development Fund. ST has
been funded by the Deutsche Forschungsgemeinschaft (DFG, German
Research Foundation) – Projektnummer 396845724. The authors would like
to acknowledge the support of the Hausdorff Institute for Mathematics,
since part of this work was completed during their stay at the
Trimester Program on kinetic theory.

\appendix

\section{Proof of Theorem~\ref{Thm:L2:convergence} on \texorpdfstring{$L^2$}{L2} convergence for
  the constant kernel}
\label{sec:aux_thm}

We gather here the proof of \cref{Thm:L2:convergence}, which is a
small modification of \cite[Lemma 6.1]{CMM10}. Since the proof is
independent of the rest of the paper and is a small improvement of the
aforementioned one, we prefer to give it in an appendix.

Our starting point is \cite[Lemma 6.1]{CMM10}:

\begin{lemma}\label{Lem:L2:convergence:CMM}
  Let $f$ be a solution to~\eqref{eq:Smol:selfsim} for the constant
  kernel $K=2$ with total mass $1$ and initial condition $f_0$ such
  that $f_0 \in L^{2}(\dx) \cap L^{1}(x^2\dx)$. Let
  $G^{0}(x)=\ee^{-x}$ be the unique stationary solution
  to~\eqref{eq:Smol:selfsim} with total mass $1$. There exist
  constants $C, T > 0$ depending only on $f_0$ such that
  \begin{equation*}
    \norm{f(t,\cdot)-G^{0}}_{L^{2}}
    \leq
    C\ee^{-\frac{1}{2}t}
    \qquad \text{for all } t \geq T.
  \end{equation*}
\end{lemma}
We want to make two modifications to this statement, namely: 1.~that a
bound can be given for all $t \geq 0$, and 2.~that the constants can
be explicitly calculated and depend only on $\norm{f_0}_{L^2}$ and
$\norm{f_0}_{L^1_2}$. The first modification is very simple, and we
give it first:
\begin{lemma}\label{Lem:L2:convergence:CMM2}
  In the conditions of Lemma \ref{Lem:L2:convergence:CMM}, there
  exists a constant $C > 0$ depending only on $f_0$ such that
  \begin{equation*}
    \norm{f(t,\cdot)-G^{0}}_{L^{2}}
    \leq
    C\ee^{-\frac{1}{2}t}
    \qquad \text{for all } t \geq 0.
  \end{equation*}
\end{lemma}

\begin{proof}
  For $t \geq T$ it is clearly true from
  \cref{Lem:L2:convergence:CMM}. For $0 \leq t \leq T$ we can use any
  available bound on the growth of the $L^2$ norm of a solution. For
  example, one can easily calculate that
  \begin{multline*}
    \frac12 \ddt \norm{f}_{L^2}^2
    =
    \frac32 \norm{f}_{L^2}^2
    + \int_0^\infty \int_0^x f(x) f(x-y) f(y) \d y \d x
    - \norm{f}_{L^2}^2 \int_0^\infty f(x)\dx
    \\
    \leq
    \frac32 \norm{f}_{L^2}^2
    + \norm{f}_{L^2}^2 \int_0^\infty f(x)\dx,
  \end{multline*}
  where we have used Cauchy-Schwarz's inequality on the integral term
  and disregarded the negative one. Since $\int f$ can be calculated
  explicitly, we see $\int f \leq \max\{1, \int f_0\} =: C_1$, so
  \begin{equation*}
    \norm{f}_{L^2}^2 \leq \norm{f_0}_{L^2}^2 \, \exp{({(3 + 2C_1) t})}
    \qquad \text{for all } t \geq 0.
  \end{equation*}
  In particular, for all $0 \leq t \leq T$,
  \begin{equation*}
    \norm{f}_{L^2}
    \leq
    \norm{f_0}_{L^2} \, \exp{\left({\big(\frac32 + C_1\big) T}\right)}
    \,
    \ee^{\frac{T}{2}} \ee^{-\frac{t}{2}}
    =
    \norm{f_0}_{L^2} \, \exp{\left({(2 + C_1) T}\right)}
    \,
    \ee^{-\frac{t}{2}},
  \end{equation*}
  so
  \begin{equation*}
    \norm{f - G^0}_{L^2}
    \leq
    \norm{f}_{L^2}
    + \norm{G^0}_{L^2}
    \leq
    \norm{f_0}_{L^2} \, \exp{\left({(2 + C_1) T}\right)}
    \,
    \ee^{-\frac{t}{2}}
    + \frac12.
  \end{equation*}
  We obtain then
  \begin{equation*}
    \norm{f(t,\cdot)-G^{0}}_{L^{2}}
    \leq
    C_2 \ee^{-\frac{1}{2}t}
    \qquad \text{for all } t \geq 0,
  \end{equation*}
  with
  \begin{equation*}
    C_2 := \max\{C,\ \frac12 + \norm{f_0}_{L^2}^2 \exp{({(2 + C_1) T})} \},
  \end{equation*}
  where $C$ and $T$ are those from \cref{Lem:L2:convergence:CMM}.
\end{proof}

The final version we want to give is \cref{Thm:L2:convergence}, which
is the same as \cref{Lem:L2:convergence:CMM2} with the addition that
the constant $C$ depends only on $\|f_0\|_{L^2}$ and
$\|f_0\|_{L^1_2}$. Let us give the proof of this:

\begin{proof}[Proof of \cref{Thm:L2:convergence}]
  We notice that for $t \in [0,T]$, the constant we obtain in the
  proof of Lemma \ref{Lem:L2:convergence:CMM2} depends only on
  $\norm{f_0}_{L^2}$ and $\int_0^\infty f_0$ (increasingly), so it can be made to
  depend only on $\norm{f_0}_{L^2}$ and $\norm{f_0}_{L^1_2}$, as we
  want. Hence we just need to check that the constants obtained in the
  proof of \cite[Lemma 6.1]{CMM10} depend only on the specified norms
  of $f_0$. One may assume that $\int f_0 = 1$, since one can always
  reduce the proof to that case by a change of variables. One can see
  from the proof in \cite{CMM10} that all constants are explicit,
  except for the one called $\eps_2$, defined by
  \begin{equation}
    \label{eq:eps2}
    \eps_2 := \inf_{|\xi| > \eps_1} |1 - \phi_0(\xi)|,
  \end{equation}
  where $\eps_1$ is a quantity that depends only on
  $\norm{f_0}_{L^1_2}$, and $\phi_0$ is the Fourier transform of
  $f_0$:
  \begin{equation*}
    \phi_0(\xi) := \int_0^\infty e^{-i \xi x} f_0(x) \d x,
    \qquad \xi \in \R.
  \end{equation*}
  (Notice that we have adapted the definition of $\eps_2$ to our
  current choice of constant kernel $K=2$ instead of $K=1$, as used in
  \cite{CMM10}; this is not essential). We need then to find an
  explicit lower bound of $\eps_2$ that depends only on
  $\eps_1$, $\|f_0\|_{L^2}$ and $\|f_0\|_{L^1_2}$. This is given
  by Lemma \ref{lem:Fourier_bound}, which we prove in the remaining
  part of this appendix.
\end{proof}

\section{An estimate on the Fourier transform}
\label{sec:fourier}

In order to estimate the $\eps_2$ in \eqref{eq:eps2} we need to
understand the following. The Fourier transform of $f_0$ is always
less than or equal to $1$ in absolute value, since $f_0 \geq 0$ with
integral $1$. Its absolute value is never equal to $1$ except for at
the mode $\xi = 0$, and we need to find a quantitative estimate of
this phenomenon. Our final result is given in
\cref{lem:Fourier_bound}, but we will need a few lemmas to arrive
there. The next one contains the central part of the argument:

\begin{lemma}
  \label{lem:local_sine_bound}
  Given $R > 0$, take a nonnegative function
  $f \in L^1(-R, R)\cap L^{2}(-R,R)$. Then
  \begin{equation*}
    \int_{-R}^{R} f(x) \sin(x) \dx \leq
    \left( 1 - \alpha \right)
    \int_{-R}^{R} f(x) \d x,
  \end{equation*}
  where
  \begin{equation*}
    \alpha := \frac{M^4}{128 \pi n^2 \norm{f}_{L^2}^4},
    \qquad
    n := 1 + \frac{R}{2 \pi},
    \qquad
    M := \int_{-R}^{R} f(x) \d x.
  \end{equation*}
\end{lemma}

\begin{proof}
  It is clearly enough to prove it in the case
  $\int_{-R}^{R} f(x) \d x = 1$, so we make this assumption
  throughout. For $0 < \eps < \frac{\pi}{2}$ to be fixed later, we
  call $A_\eps$ the $\eps$-neighbourhood of the points in
  $[-R, R]$ where $\sin x = 1$:
  \begin{equation*}
    A_\eps := \Big\{x \in [-R, R] \mid
    \text{
      $\Big|x - \frac{(4k + 1)\pi}{2} \Big| < \eps$ for some odd integer $k$
    }
    \Big\},
  \end{equation*}
  and $B_\eps$ its complement in $[-R,R]$:
  \begin{equation*}
    B_\eps := [-R,R] \setminus A_\eps.
  \end{equation*}
  Since $B_\eps$ does not contain the points where $\sin x = 1$,
  there is a positive function $\delta = \delta(\eps)$ such that
  \begin{equation*}
    \int_{B_\eps} f(x) \sin x \d x
    \leq (1 - \delta(\eps)) \int_{B_\eps} f(x) \d x.
  \end{equation*}
  For example, \cref{lem:delta-epsilon-estimate} gives a simple
  explicit bound of $\delta(\eps)$. For convenience, we call
  \begin{equation*}
    m_A := \int_{A_\eps} f(x) \d x,
    \qquad
    m_A := \int_{B_\eps} f(x) \d x = 1 - m_A.
  \end{equation*}
  Hence,
  \begin{multline*}
    \int_{-R}^{R} f(x) \sin x \d x
    = \int_{A_\eps} f(x) \sin x \d x
    + \int_{B_\eps} f(x) \sin x \d x
    \\
    \leq m_A + (1-\delta(\eps)) m_B
    = m_A + (1-\delta(\eps)) (1-m_A)
    = 1 - \delta(\eps) (1 - m_A).
  \end{multline*}
  Now, by Cauchy-Schwarz's inequality we notice that
  \begin{equation*}
    m_A = \int_{A_\eps} f(x) \d x
    \leq
    \|f\|_{L^2} \sqrt{ |A_\eps| }
    \leq  \|f\|_{L^2} \sqrt{2 \big(1 + \frac{R}{2 \pi}\big) \eps},
  \end{equation*}
  since the Lebesgue measure of $A_\eps$ is at most
  $2 \big(1 + \frac{R}{2 \pi}\big) \eps$. For convenience, call
  $n := 1 + \frac{R}{2 \pi}$. We then choose $\eps$ such that
  \begin{equation*}
    \|f\|_{L^2} \sqrt{2 n \eps} = \frac12,
    \qquad \text{that is,} \qquad
    \eps := (8 n \|f\|_{L^2}^2)^{-1}
  \end{equation*}
  and we obtain 
  \begin{equation*}
    \int_{-R}^{R} f(x) \sin x \d x
    \leq 1 - \frac12 \delta(\eps).
  \end{equation*}
  Using our bound of $\delta(\eps)$ from
  \cref{lem:delta-epsilon-estimate} we finally obtain the result.
\end{proof}

\begin{lemma}
  \label{lem:delta-epsilon-estimate}
  For every $0 < \eps < \pi/2$,
  \begin{equation*}
    \delta(\eps) := 1 - \sup_{0 < x < \frac{\pi}{2} - \eps}
    \big( \sin x \big)
    \geq \frac{\eps^2}{\pi}.
  \end{equation*}
\end{lemma}

\begin{proof}
  It is easy to check that
  \begin{equation*}
    \sin x \leq 1 - \frac{\left( x - \frac{\pi}{2} \right)^2 }{\pi}
    \qquad \text{for all $0 \leq x \leq \frac{\pi}{2}$},
  \end{equation*}
  which easily implies the statement.
\end{proof}

\begin{lemma}
  \label{lem:sine_bound}
  Take a nonnegative function $f \in L^1(\R) \cap L^2(\R)$ with
  $\int_{-\infty}^\infty |x| f(x) \d x < +\infty$. Then
  \begin{equation*}
    \int_{-\infty}^{+\infty} f(x) \sin(x) \dx \leq
    \left( 1 - \alpha \right)
    \int_{-\infty}^{+\infty} f(x) \d x,
  \end{equation*}
  where
  \begin{equation*}
    \alpha := \frac{M^6}{2^{17} R^2 \norm{f}_{L^2}^4},
    \qquad
    M := \int_{-\infty}^{+\infty} f(x) \d x,
    \qquad
    R := \max\Big\{M, 2 \int_{-\infty}^{+\infty} |x| f(x) \d x \Big\}.
  \end{equation*}
  (In the trivial case that $f = 0$, it is understood that the right
  hand side is also $0$.)
\end{lemma}

\begin{proof}
  Again, it is clearly enough to prove it when
  $M = \int_{-\infty}^\infty f = 1$, so let us assume this. Call
  \begin{equation*}
    K := \int_{-\infty}^{+\infty} |x| f(x) \d x.
  \end{equation*}
  If we take
  any $R \geq 2 K$, then
  \begin{equation*}
    \int_{|x| > R} f(x) \d x
    \leq \frac{1}{2 K} \int_{\R} |x| f(x) \d x
    = \frac{K}{2 K} = \frac12.
  \end{equation*}
  Now, call
  \begin{equation*}
    m_R := \int_{|x| < R} f(x) \d x.
  \end{equation*}
  From the previous bound we know $m_R \geq \frac12$. Now, using
  \cref{lem:local_sine_bound} we have:
  \begin{multline*}
    \int_{-\infty}^{\infty} f(x) \sin x \d x
    =
    \int_{|x| \leq R} f(x) \sin x \d x
    +
    \int_{|x| > R} f(x) \sin x \d x
    \\
    \leq
    (1-\alpha_1) \int_{|x| \leq R} f(x) \d x
    +
    \int_{|x| > R} f(x) \d x
    \\
    = (1 - \alpha_1) m_R + (1- m_R)
    = 1 - \alpha_1 m_R,
  \end{multline*}
  where
  \begin{equation*}
    \alpha_1 := \frac{m_R^4}{128 \pi n^2 \norm{f}_{L^2(-R,R)}^4},
    \qquad
    n := 1 + \frac{R}{2 \pi},
  \end{equation*}
  In order to simplify the expression, take $R =: \max\{1, 2K\}$, so
  that $R \geq 1$ is ensured. Then $n \leq 2 R$ and
  \begin{equation*}
    \alpha_1 m_R
    \geq
    \frac{m_R^5}{2^{11} R^2 \norm{f}_{L^2}^4}
    \geq
    \frac{1}{2^{16} R^2 \norm{f}_{L^2}^4},
  \end{equation*}
  since we know $m_R \geq 1/2$.
\end{proof}

Now we can complete our main bound, used in the proof of
\ref{Thm:L2:convergence}:

\begin{lemma}
  \label{lem:Fourier_bound}
  Take a nonnegative function $f \in L^1(\R) \cap L^2(\R)$ with
  $\int_{-\infty}^\infty |x| f(x) \d x < +\infty$. Then for all $\xi
  \in \R$ we have
  \begin{equation*}
    \left|
      \int_{-\infty}^{+\infty} f(x) e^{-i x \xi} \dx
    \right|
    \leq
    \left( 1 - \alpha \right)
    \int_{-\infty}^{+\infty} f(x) \d x,
  \end{equation*}
  where
  \begin{gather*}
    \alpha = \alpha(f, \xi) := \frac{\xi^2 M^6}{2^{16} R^2 \norm{f}_{L^2}^4},
    \qquad
    M := \int_{-\infty}^{+\infty} f(x) \d x,
    \qquad
    R := 2 \xi \int_{-\infty}^{+\infty} |x| f(x) \d x + 2 \pi M.
  \end{gather*}
  (In the trivial case that $f = 0$, it is understood that the right
  hand side is also $0$.)
\end{lemma}

\begin{proof}
  By the change of variables $y \mapsto x \xi$, it is enough to
  prove it when $\xi = 1$. By scaling $f$ as before, we may also
  assume that $M = \int_{-\infty}^\infty f = 1$. We use the following
  trick to rewrite the modulus as an integral similar to that in
  \cref{lem:sine_bound}: if we call
  \begin{equation*}
    a := \int_{-\infty}^\infty f(x) \cos x \d x,
    \qquad
    b := \int_{-\infty}^\infty f(x) \sin x \d x,
  \end{equation*}
  then $a^2 + b^2 = 1$ and there exists some $\theta \in [0, 2 \pi)$
  such that $a = \sin \theta$, $b = \cos \theta$.
  \begin{multline*}
    \left|
      \int_{-\infty}^{+\infty} f(x) e^{-i x} \dx
    \right|^2
    =
    \left( \int_{-\infty}^{+\infty} f(x) \cos x \dx  \right)^2
    +
    \left( \int_{-\infty}^{+\infty} f(x) \sin x \dx \right)^2
    \\
    =
    \sin \theta \int_{-\infty}^{+\infty} f(x) \cos x \dx
    +
    \cos \theta \int_{-\infty}^{+\infty} f(x) \sin x \dx
    \\
    =
    \int_{-\infty}^{+\infty} f(x) \sin (x+ \theta) \dx
    =
    \int_{-\infty}^{+\infty} f(x-\theta) \sin x \dx.
  \end{multline*}
  We then apply \cref{lem:sine_bound} to $\tilde{f}(x) := f(x-\theta)$ to
  obtain the result. Notice that $\|\tilde{f}\|_2 = \|f\|_2$,
  $\int_{-\infty}^\infty \tilde{f} = \int_{-\infty}^\infty f$, and
  \begin{equation*}
    \int_{-\infty}^\infty |x| \tilde{f}(x) \d x
    =
    \int_{-\infty}^\infty |x+\theta| f (x) \d x
    \leq
    \int_{-\infty}^\infty |x| f (x) \d x + 2 \pi M.
    \qedhere
  \end{equation*}
\end{proof}

\section{Proof of Lemma~\ref{Lem:diff:sg:NEW} on the transport semigroup}
\label{Sec:proof:sg}

\begin{proof}[Proof of Lemma~\ref{Lem:diff:sg:NEW}]
  We first show that for each $t\geq 0$ the operator $T_{t}$ is well-defined on the respective spaces, while for $H^{-1}(\ee^{\mu x})$ we also recall Remark~\ref{Rem:lin:op:H1}. In fact, using the change of variables $x\mapsto x\ee^{-t}$ we find for $L^{1}_{k}$ that
  \begin{multline*}
    \norm{T_{t}h}_{L^1_k}=\int_{0}^{\infty}\abs{h(x\ee^{t})}(1+x)^{k}\dx=\ee^{-t}\int_{0}^{\infty}\abs{h(x)}(1+x\ee^{-t})^{k}\dx\\*
    \leq \ee^{-t}\int_{0}^{\infty}\abs{h(x)}(1+x)^{k}\dx=\ee^{-t}\norm{h}_{L^1_k}.
  \end{multline*}
  Similarly, we get
  \begin{multline*}
    \norm{T_{t}h}_{L^{2}(\ee^{\mu x})}^2=\int_{0}^{\infty}\abs{h(x\ee^{t})}^2\ee^{\mu x}\dx=\ee^{-t}\int_{0}^{\infty}\abs{h(x)}^2\ee^{\mu x\ee^{-t}}\dx\\*
    \leq \ee^{-t}\int_{0}^{\infty}\abs{h(x)}^{2}\ee^{\mu x}\dx=\ee^{-t}\norm{h}_{L^{2}(\ee^{\mu x})}^{2}.
  \end{multline*}
  Finally, for $H^{-1}(\ee^{\mu x})$ we obtain
  \begin{multline*} 
    \norm{T_{t}h_{k}}_{H^{-1}(\ee^{\mu x})}^{2}=\int_{0}^{\infty}\ee^{\mu x}\biggl(\int_{x}^{\infty}h_{k}(z\ee^{t})\dz\biggr)^2\dx=\ee^{-2t}\int_{0}^{\infty}\ee^{\mu x}\biggl(\int_{x\ee^{t}}^{\infty}h_{k}(z)\dz\biggr)^{2}\dx\\*
    =\ee^{-3t}\int_{0}^{\infty}\ee^{\mu x\ee^{-t}}\biggl(\int_{x}^{\infty}h_{k}(z)\dz\biggr)^{2}\dx\\*
    \leq \ee^{-3t}\int_{0}^{\infty}\ee^{\mu x}\biggl(\int_{x}^{\infty}h_{k}(z)\dz\biggr)^{2}\dx=\ee^{-3t}\norm{h_{k}}_{H^{-1}(\ee^{\mu x})}^{2}.
  \end{multline*}
  In particular, this yields the estimates 
  \begin{equation}\label{eq:sg:diff:ubound}
    \begin{aligned}
      \norm{T_{t}}_{L^{1}_{k}\to L^{1}_{k}}&\leq \ee^{-t}\\
      \norm{T_{t}}_{L^{2}(\ee^{\mu x})\to L^{2}(\ee^{\mu x})}&\leq \ee^{-\frac{1}{2}t}\\
      \norm{T_{t}}_{H^{-1}(\ee^{\mu x})\to H^{-1}(\ee^{\mu x})}&\leq \ee^{-\frac{3}{2}t}
    \end{aligned}
  \end{equation}
  for all $t\geq 0$.

  It thus remains to verify the strong continuity. By means of \cite[Ch.\@I, 5.3 Proposition]{EnN00} and~\eqref{eq:sg:diff:ubound} it is sufficient to show that $\lim_{t\to 0}T_{t}h=h$ for all $h$ in a dense subset $D\subset L^{1}((1+x)^{k})$, $D\subset L^{2}(\ee^{\mu x})$ or $D\subset H^{-1}(\ee^{\mu x})$ respectively. Thus, taking for example $D=C_{c}^{\infty}((0,\infty))$ we find for $h\in D$ that
  \begin{multline*}
    \norm{T_{t}h-h}_{L^1_k}=\int_{0}^{\infty}\abs{h(x\ee^{t})-h(x)}(1+x)^{k}\dx=\int_{0}^{\infty}\abs*{\int_{0}^{t}\del_{\tau}h(x\ee^{\tau})\dd{\tau}}(1+x)^{k}\dx\\*
    =\int_{0}^{\infty}\abs*{\int_{0}^{t}x\ee^{\tau}h'(x\ee^{\tau})\dd{\tau}}(1+x)^{k}\dx\leq \int_{0}^{t}\ee^{\tau}\int_{0}^{\infty}\abs{h'(x\ee^{\tau})}x(1+x)^{k}\dx\dd{\tau}\\*
    =\int_{0}^{\infty}\ee^{-\tau}\int_{0}^{\infty}\abs{h'(x)}x(1+x\ee^{-\tau})^{k}\dx\dd{\tau}\leq \norm{xh'(x)}_{L^1_k}(1-\ee^{-t}).
  \end{multline*}
  For $t\to0$ the right-hand side converges to zero which finally yields the strong continuity for $L^{1}((1+x)^{k})$. For $L^{2}(\ee^{\mu x})$ we argue similarly and get
  \begin{multline*}
    \norm{T_{t}h-h}_{L^{2}(\ee^{\mu x})}^{2}=\int_{0}^{\infty}\abs{h(x\ee^{t})-h(x)}^2\ee^{\mu x}\dx=\int_{0}^{\infty}\abs*{\int_{0}^{t}\del_{\tau}h(x\ee^{\tau})\dd{\tau}}^2\ee^{\mu x}\dx\\*
    =\int_{0}^{\infty}\abs*{\int_{0}^{t}x\ee^{\tau}h'(x\ee^{\tau})\dd{\tau}}^2\ee^{\mu x}\dx.
  \end{multline*}
  Together with Cauchy-Schwartz and the change of variables $x\mapsto x\ee^{-\tau}$, we find
  \begin{multline*}
    \norm{T_{t}h-h}_{L^{2}(\ee^{\mu x})}^{2}\leq \int_{0}^{\infty}t\int_{0}^{t}x^2 \ee^{2\tau}\abs{h'(x\ee^{\tau})}^{2}\dd{\tau}\ee^{\mu x}\dx=t\int_{0}^{t}\ee^{2\tau}\int_{0}^{\infty}x^2\abs{h'(x\ee^{\tau})}^2\ee^{\mu x}\dx\dd{\tau}\\*
    \leq t\int_{0}^{t}\ee^{-\tau}\int_{0}^{\infty}x^2\abs{h'(x)}\ee^{\mu x}\dx\dd{\tau}=\norm{xh'(x)}_{L^{2}(\ee^{\mu x})}t(1-\ee^{-t}).
  \end{multline*}
  Again, the right-hand side converges to zero as $t\to 0$ which proves the strong continuity also for $L^{2}(\ee^{\mu x})$. Finally, for $H^{-1}(\ee^{\mu x})$ we get analogously
  \begin{multline*}
    \norm{T_{t}h-h}_{H^{-1}(\ee^{\mu x})}^{2}=\int_{0}^{\infty}\ee^{\mu x}\biggl(\int_{x}^{\infty}h(z\ee^{t})-h(z)\dz\biggr)^{2}\dx\\*
    =\int_{0}^{\infty}\ee^{\mu x}\biggl(\int_{x}^{\infty}\int_{0}^{t}\del_{\tau}h(z\ee^{\tau})\dd{\tau}\dz\biggr)^{2}\dx=\int_{0}^{\infty}\ee^{\mu x}\biggl(\int_{x}^{\infty}\int_{0}^{t}z\ee^{\tau}h'(z\ee^{\tau})\dd{\tau}\dz\biggr)^{2}\dx\\*
    =\int_{0}^{\infty}\ee^{\mu x}\biggl(\int_{0}^{t}\ee^{-\tau}\int_{x\ee^{\tau}}^{\infty}zh'(z)\dz\dd{\tau}\biggr)^{2}\dx.
  \end{multline*}
  Applying Hölder's inequality we find
  \begin{multline*}
    \norm{T_{t}h-h}_{H^{-1}(\ee^{\mu x})}^{2}\leq \int_{0}^{\infty}\ee^{\mu x}\biggl(\int_{0}^{t}\ee^{-2\tau}\dd{\tau}\biggr)\biggl(\int_{0}^{t}\biggl(\int_{x\ee^{\tau}}^{\infty}zh'(z)\biggr)^{2}\dd{\tau}\biggr)\dx\\*
    =\frac{1}{2}(1-\ee^{-2t})\int_{0}^{\infty}\ee^{\mu x}\int_{0}^{t}\biggl(\int_{x\ee^{\tau}}^{\infty}zh'(z)\dz\biggr)^{2}\dd{\tau}\dx.
  \end{multline*}
  Since we are interested in the limit $t\to 0$, we can assume that $t\leq 1$ and thus, for fixed $h\in C_{c}^{\infty}(0,\infty)$ the integral on the right-hand side is bounded. Therefore, for $t\to 0$ the right-hand side converges to zero which proves the strong continuity also for $H^{-1}(\ee^{-\mu x})$.
  
  To determine the generator, we take $h\in C_{c}^{\infty}(0,\infty)$ and compute
  \begin{equation*}
    \lim_{t\to 0}\frac{1}{t}(T_{t}h-h)(x)=\lim_{t\to 0}\frac{1}{t}\bigl(h(x\ee^{t})-h(x)\bigr)=xh'(x).
  \end{equation*}
  This shows that $C_{c}^{\infty}(0,\infty)\subset D(\B_{1})$ and $\B_{1}|_{C_{c}^{\infty}(0,\infty)}=x\del_{x}$. Thus, to conclude the proof it suffices to prove that $C_{c}^{\infty}$ is a core for $\B_{1}$.
  
  According to \cite[Ch.\@ I, 1.7 Proposition]{EnN00}, it suffices to verify that $C_{c}^{\infty}(0,\infty)$ is invariant under the action of $(T_{t})_{t\geq 0}$ and that $C_{c}^{\infty}(0,\infty)$ is dense in $L^{1}((1+x)^{k})$, $L^{2}(\ee^{\mu x})$ and $H^{-1}(\ee^{\mu x})$, respectively. Due to the explicit formula $(T_{t}h)(x)=h(x\ee^{t})$ the invariance is clear while density is also well-known or clear by construction.
  
  The claim for the spaces
  $L^{1}((1+x)^{k})\cap\{\int_{0}^{\infty}xf(x)\dx=0\}$,
  $L^{2}(\ee^{\mu x})\cap\{\int_{0}^{\infty}xf(x)\dx=0\}$ and
  $H^{-1}(\ee^{\mu x})\cap\{\int_{0}^{\infty}xf(x)\dx=0\}$ directly
  follows by restricting the semigroup once we notice that
  $(T_{t})_{t\geq 0}$ preserves the constraint.
\end{proof}

\bibliographystyle{plain}

\end{document}